\documentclass[11pt]{article}
\usepackage[utf8]{inputenc}
\usepackage{amsmath,amsthm,amsfonts,fullpage, amssymb, xcolor,hyperref,bbm,graphicx}
\usepackage{authblk}

\newcommand{\Poi}{\operatorname*{\mathrm{Poi}}}
\newcommand{\Exp}{\operatorname*{\mathbb{E}}}

\newcommand{\poi}[2]{\text{poi}(#1,#2)}
\newcommand{\poii}[1]{\poi{#1}{i}}
\newcommand{\hist}[2]{h^{#1}_{#2}}
\newcommand{\histp}{\hist{P}{y_j}}
\newcommand{\histq}{\hist{Q|y_j}{x}}

\newcounter{nTheorems}

\newtheorem{theorem}[nTheorems]{Theorem}
\newtheorem{corollary}[nTheorems]{Corollary}

\newtheorem{lemma}[nTheorems]{Lemma}

\newtheorem{assumption}[nTheorems]{Assumption}

\theoremstyle{definition}
\newtheorem{example}{Example}
\newtheorem{definition}[nTheorems]{Definition}

\newcommand{\eps}{\varepsilon}

\newcommand{\authnote}[2]{[{\color{blue}\textbf{#1:} #2}]}
\newcommand{\chernoff}{\stackrel{\text{Ch}}{\leq}}

\newcommand{\trung}[1]{\authnote{Trung}{#1}}
\newcommand{\paul}[1]{\authnote{Paul}{#1}}
\newcommand{\hongao}[1]{{\color{cyan} [\textbf{Hongao:} #1]}}

\title{Improving Pearson’s chi-squared test: hypothesis testing of distributions - optimally\footnote{Work partially funded by NSF award CCF-2127806}}

\author[1,2]{Trung Dang}
\author[1,3]{Walter McKelvie}
\author[1]{Paul Valiant}
\author[1]{Hongao Wang}
\affil[1]{Purdue University}
\affil[2]{UT Austin}
\affil[3]{Columbia University}

\begin{document}
\maketitle

\begin{abstract}
Pearson's chi-squared test, from 1900, is the standard statistical tool for ``hypothesis testing on distributions": namely, given samples from an unknown distribution $Q$ that may or may not equal a hypothesis distribution $P$, we want to return ``yes" if $P=Q$ and ``no" if $P$ is far from $Q$. While the chi-squared test is easy to use, it has been known for a while that it is not ``data efficient", it does not make the best use of its data. Precisely, for accuracy $\eps$ and confidence $\delta$, and given $n$ samples from the unknown distribution $Q$, a tester should return ``yes" with probability $>1-\delta$ when $P=Q$, and ``no" with probability $>1-\delta$ when $|P-Q|>\eps$. The challenge is to find a tester with the \emph{best} tradeoff between $\eps$, $\delta$, and $n$.

We introduce a new tester, efficiently computable and easy to use, which we hope will replace the chi-squared tester in practical use. Our tester is found via a new non-convex optimization framework that essentially seeks to ``find the tester whose Chernoff bounds on its performance are as good as possible". This tester is $1+o(1)$ optimal, in that the number of samples $n$ needed by the tester is within $1+o(1)$ factor of the samples needed by \emph{any} tester, even non-linear testers (for the setting: accuracy $\eps$, confidence $\delta$, and hypothesis $P$). We complement this algorithmic framework with matching lower bounds saying, essentially, that ``our tester is instance-optimal, even to $1+o(1)$ factors, to the degree that Chernoff bounds are tight". %The new estimator we introduce is semilinear for fixed $q$, in the sense that the tester is linear \emph{across} domain elements, though applies an arbitrary look-up table for each domain element. The coefficients, for a given domain element, have an interesting ``log hyperbolic cosine" form that naturally generalizes many previous estimators. 
Our overall non-convex optimization framework %---for finding estimators with optimal Chernoff bounds---
extends well beyond the current problem and is of independent interest.
\end{abstract}
% {\Large Todos:

% ??? (The locations where we have question marks)
% \begin{itemize}
%     \item Before Equation~\ref{eq:relaxed-bound}. \trung{resolved}
%     \item Before Equation~\ref{eq:relaxed-bound2}
%     \item Before Lemma~\ref{lem:log-concave-mixture}
%     \item In the proof of Lemma~\ref{lem:two-point}
%     \item In the Lemma~\ref{lem:derivative-x}
%     \item In the proof of Lemma~\ref{lem:derivative-u}
%     \item In the proof of Lemma~\ref{lem:derivative-analysis}
% \end{itemize}

% Notation macro define
% \begin{itemize}
%     \item poisson: we have $\poii{ky_j}$ and $\poi{\lambda}{k}$. not finish the change.
%     \item define some notation macros for histogram $h$?
%     \item change the statement of the lemmas such that they will match each other and extract the final main theorem that asserts the upper bound and lower bound match each other and get all of the lemmas ready for this theorem, maybe we need the extended minimax theorem we proved here to provide the optimality of that matched bound (or maybe the upper bound).
% \end{itemize}

\section{Introduction}
In this paper we consider the statistical problem of ``hypothesis testing of distributions'', and provide a  principled yet practical new approach to optimal performance. The standard approach in practice is Pearson's chi-squared test (described below), invented in 1900, and remaining one of the most widely used tests in statistics, a bedrock of the field. Nonetheless, the chi-squared test was invented before the advent of principled methodical approaches to data-efficient statistics, and from this perspective looks a bit ad-hoc. Since 1900, many post-hoc justifications of the chi-squared test have been developed, as well as ``patches" to tailor its performance to certain situations. 

The most standard TCS model that captures many of the use cases of the chi-squared test is as follows: we start with a known \emph{hypothesis} distribution $P$ of support size $n$; we receive $k$ samples from an unknown distribution $Q$, and wish to return a ``yes" or ``no" answer distinguishing, as well as possible, whether our hypothesis $P$ is correct---namely $P=Q$---versus far from correct. Explicitly, given an $\ell_1$ distance bound $\eps$, and a failure probability $\delta$, we want to say ``yes" with probability $\geq 1-\delta$ when the samples are from $P$, and we want to say ``no" with probability $\geq 1-\delta$ when the samples are from \emph{any} distribution $Q$ such that $||P-Q||_1\geq\eps$. The goal is to find the algorithm with the best tradeoff between the number of samples $k$, the accuracy threshold $\eps$, and the confidence $1-\delta$, for any given hypothesis $P$.

There are thus 3 equivalent ways of viewing our problem: given $k$ data points and an allowed failure rate $\delta$, find the tester that optimizes the accuracy $\epsilon$; alternatively, given $k$ data points and a required accuracy $\epsilon$, find the tester with the smallest failure rate $\delta$; or, finally, given an allowed failure probability $\delta$ and a required accuracy $\epsilon$, find the minimum amount of data $k$ needed. Algorithms for any of these formulations can be transparently reparameterized/reformulated to solve the problem from the other perspectives (via a black-box binary search reduction).

%??? This model has been well-studied [argue for examples of uses]

We set out to solve this bedrock problem of statistics, in a way that is practical, robust, and extensible. 
\begin{figure}
    \centering
    \includegraphics[width=0.49\columnwidth]{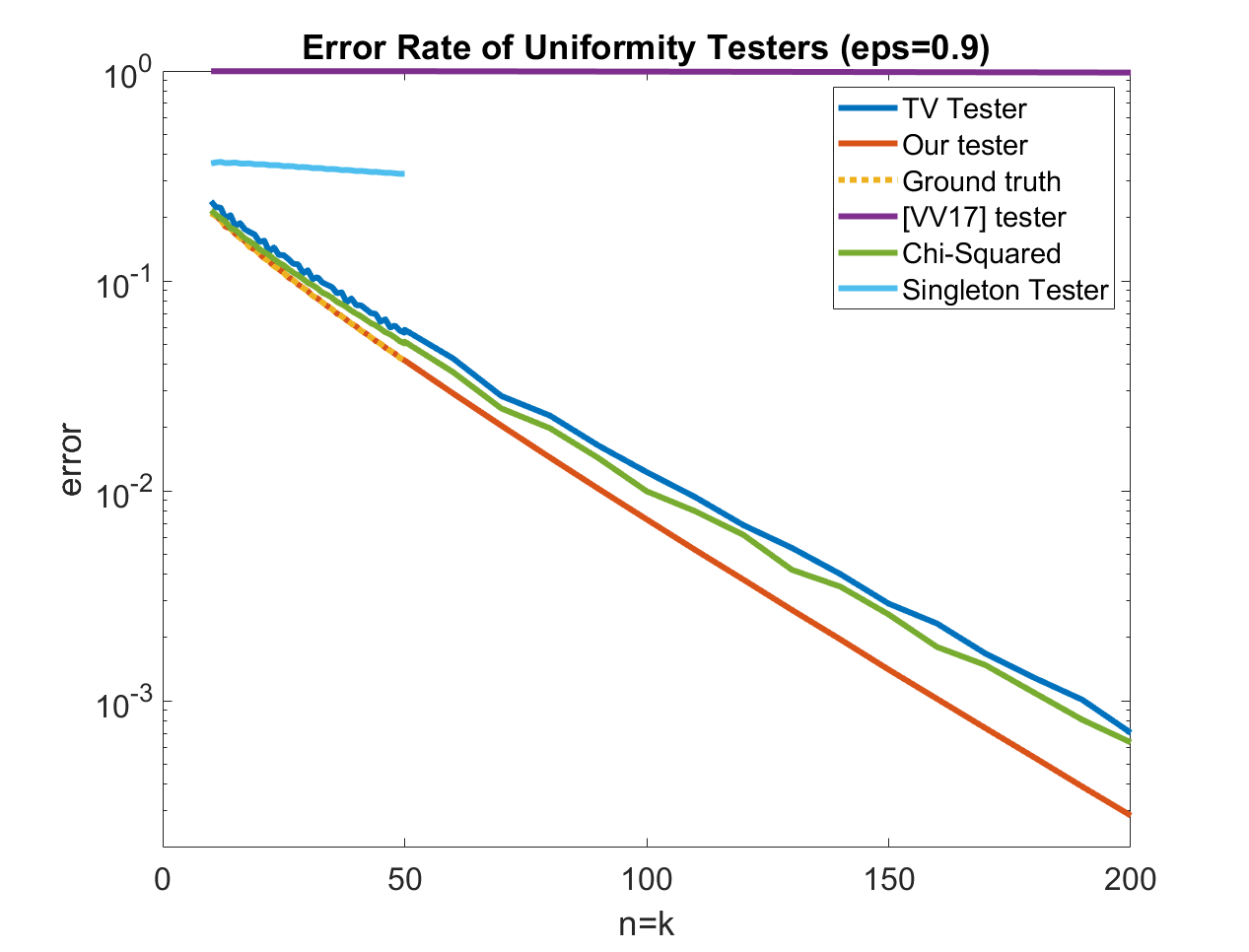} \includegraphics[width=0.49\columnwidth]{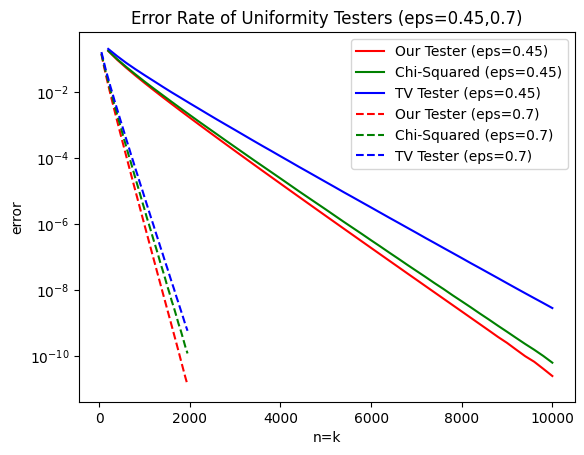} 
    \caption{\textit{Performance of Various Uniformity Testers. Plots show the result of numerical experiments comparing the performance of our new tester to 3 standard testers, that respectively threshold the chi-squared metric, the total variation distance, and the number of singleton elements in the data. The left plot shows $\eps=0.9$, with error $\delta$ (maximum of Type I and Type II errors) plotted as a function of $n=k$. The yellow dotted line, expressing a brute force computation of the ground truth optimal tester, is visually synonymous with our tester for the range of parameters for which we could compute it, $k=n\leq 50$. The right plot shows similar plots for $\eps=0.7$, and $\eps=0.45$, for larger ranges of $n=k$. In all cases in this paper, our tester was optimized and computed in $<1$s in Matlab. The choice of \emph{threshold} for each tester significantly affects its performance, and in all cases (except [VV17] which specifies its thresholds) we were as generous as possible, choosing the threshold of lowest error rate.}}
	\label{fig:uniform1}
\end{figure}

\medskip\noindent{\bf Our algorithm:} 
This paper presents a new algorithm for this problem, arising from a novel optimization framework. Our algorithm makes optimal use of its data, even up to $1+o(1)$ factors: phrased in terms of optimizing the failure probability in terms of a given number of samples $k$ and accuracy $\epsilon$, we show that the log failure rate $\log\frac{1}{\delta}$ of the tester returned by our algorithm converges to within $1+o(1)$ factor of optimal (see Theorem~\ref{thm:main}). This is not just an instance-optimal result, but a sub-constant-factor tight result. Complementing the theoretical analysis, simulations show that our tester outperforms all other testers in a variety of settings. Of particular interest is the simplest (and thus most widely studied) setting of the \emph{uniform} distribution hypothesis. Despite our tester being designed for far more general settings, we significantly outperform all other testers proposed for this iconic setting---most notably the total variation (TV) tester, the collisions tester (which is synonymous with the chi-squared tester in this special case), and the singletons tester. In particular, for small uniform distributions, we can derive, via brute force numerics, the ground-truth optimal tester; as seen in Figure \ref{fig:uniform1}, the performance of our algorithm is visually \emph{indistinguishable} from the optimal. The fact that our tester is visually indistinguishable from optimal on small inputs is particularly striking given that we analytically expect its performance to get ever more optimal for larger inputs.

While the chi-squared estimator---and many variants---perform fairly well on uniform distributions, the nonuniform case is far more challenging and exposes giant performance gulfs. See Figure \ref{fig:nonuniform}, showing a massive practical advantage of our tester over chi-squared. When $n=80$ and $k=40$, effectively no meaningful information can be learned using a chi-squared test, yet our test's error rate remains less than $10\%$.

\begin{figure}
    \centering
    \includegraphics[scale=0.5]{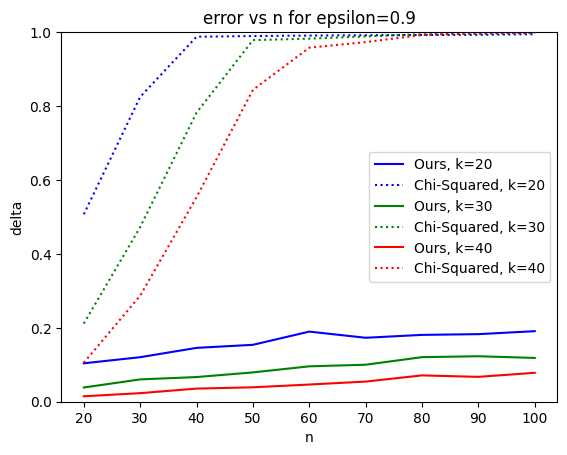}
    \caption{Performance of Testers on a Nonuniform Distribution with a Single Heavy Element. \emph{Plots show performance of testers on distributions composed of one element with weight $1 \over 2$ and $n$ elements each with weight $\frac{1}{2n}$, with $\eps = 0.9$ and $k$ ranging over $20$, $30$, or $40$.}}
	\label{fig:nonuniform}
\end{figure}

\medskip\noindent{\bf Our new algorithmic approach:} 
Our algorithmic approach is surprisingly direct: for the problem of hypothesis testing on distributions, we consider \emph{all} testers from a flexible and general class---called ``semilinear testers'', see Definition~\ref{def:semilinear}---and return the \emph{best} one, via an optimization setup. 

The direct approach is daunting, because describing the performance of an arbitrary algorithm---let alone optimizing it---is already challenging. For an algorithm to be effective, its error rate must be as small as possible, where by statistical convention, we say that a ``type I error" is when we accidentally reject the hypothesis $P=Q$ when it is true, and a ``type II error'' is when we accidentally accept the hypothesis when $|P-Q|\geq\eps$. The failure rate of our tester is the maximum failure over these two cases. Within the case of type II error, there is a third nested optimization problem, we we must necessarily consider the worst-case error over \emph{all} distributions $Q$ that are $\eps$-far from our hypothesis. Finally, even fixing a certain tester, a certain error-type, and a certain alternative distribution $Q$, we bound the failure rate via a Chernoff bound, which is a 4th level of optimization. In short, we cleanly capture the challenge of designing an optimal tester via a 4-level non-convex optimization setup, which on its surface is not encouraging.

The main point of the paper is that, despite appearances, our optimization problem is 1) algorithmically tractable, admitting effective code; 2) partially analytically tractable, showing that optimal testers have an intriguing ``$\log \cosh$" shape; 3) naturally provides matching lower bounds, showing that even though we only optimize over semilinear algorithms, we asymptotically achieve the best performance possible from \emph{any} algorithm. The generality of this new framework suggests that such an optimization approach might provide canonical optimal solutions for a much wider array of related problems that we have not begun to explore yet.

\medskip\noindent{\bf Semilinear testers:}
We now introduce the class of ``semilinear" estimators and testers. Semilinear estimators are simple yet flexible, expressing for each domain element, an arbitrary function of the number of times this element is sampled, while summing linearly across different domain elements. The goal of this paper is to find the ``right" tester from this class, instead of an ad-hoc tester from a larger class.

\begin{definition}\label{def:semilinear}
Given a sample space (domain) indexed by $j$, a semilinear estimator is represented by a table with coefficients $c_{i,j}$. Given several samples from a distribution on this domain, we represent the samples by their histogram, with $s_j$ counting the number of times element $j$ was sampled. The semilinear estimator $\mathbf{c}$, on sample $\mathbf{s}$, will return $\sum_j c_{s_j,j}$. Namely, for each domain element $j$, the number $s_j$ records the number of times this element occured in the sample, and we look up, in the $j^\textrm{th}$ column of our table $\mathbf{c}$, its $s_j^\textrm{th}$ entry, and add up all these $c_{s_j,j}$. A semilinear \emph{threshold} tester consists of a semilinear estimator and a threshold $\gamma$: it  returns ``yes" if the semilinear estimate is below the threshold $\gamma$, else ``no". Without loss of generality, we may set the threshold $\gamma$ to be 0.
\end{definition}

\begin{example}
The Pearson chi-squared test, for a given hypothesis $P$, and given $k$ samples from some distribution $Q$, represented as a histogram $\mathbf{s}$, computes the estimate \[\sum_j \frac{(s_j-k P_j)^2}{P_j}\] and returns ``yes" if it is below some threshold. Namely, it computes the squared difference between the number of times each domain element is seen, versus its expectation (if the hypothesis is true), normalized by its expectation. The Pearson chi-squared test is thus a semilinear tester.
\end{example}

As a related example (discussed more below), the uniformity tester of Gupta and Price~\cite{GP22} computes the ``Huber statistic" of $s_j-k P_j$ instead of simply squaring it, but since this function can be represented as a lookup table for each potential value of $s_j$, this statistic is also semilinear. By contrast, the instance-optimal tester of Valiant and Valiant~\cite{VV17} is not quite semilinear, because it essentially computes \emph{two} semilinear threshold tests and returns the OR of them.

\subsection{Main Results}
The main results are an algorithmic upper bound for hypothesis testing on distributions, and a lower bound that matches it, provided that ``Chernoff bounds are tight". While there is a long and celebrated history in statistics of analyzing ``reverse" Chernoff bounds, under the name ``the large deviations principle", these tools typically give insight only in the limit as the number of samples goes to infinity, without commenting on the rate of convergence; and without further foundational statistical work in this direction, our lower bounds are necessarily also of this flavor. See Section~\ref{sec:gartner-ellis} for more discussion.

Towards this end, we introduce the limit that we use for our lower bounds: we start with a fixed hypothesis distribution $P$ of support size $n$, and consider taking $\Poi(k)$ samples from it; in the limit, we essentially repeat this process $s$ times, for some positive integer $s\rightarrow\infty$. Technically, this corresponds to taking $\Poi(k s)$ samples from a version of $P$ that has each domain element subdivided into $s$ identical copies of itself: 
\begin{definition}
Given a distribution $P$ supported on $n$ elements, define $P^{sub(s)}$, read as ``$P$ subdivided $s$ times" to be the distribution supported on $n s$ elements, where each domain element of $P$ is subdivided into $s$ new domain elements of $P^{sub(s)}$, each of $\frac{1}{s}$ times the original probability mass.
\end{definition}

Our main results are about the optimum of the overall optimization problem, which we denote $\Delta(P,k,\eps)$:
\begin{definition}
%\paul{figure out how/where to define the histogram associated with P}
Given a distribution $P$, a bound $\eps$ on the $\ell_1$ distance that we wish to test for, and a number of samples $k$, let $\Delta(P,k,\eps)$ be the optimal objective value of the nonconvex optimization problem of Equation~\ref{eq:relaxed}.
\end{definition}

Our main results---including both upper and lower bounds---are summarized in the following theorem.
\begin{theorem}\label{thm:main}%\paul{``quasilinear" is aggressive, and we aren't proving this...}
    Given a distribution $P$ of support size $n$, a bound $\eps$ on the $\ell_1$ distance that we wish to test for, and a number of samples $k$, then:
    \begin{itemize}
        \item In polynomial time, we can compute $\Delta(P,k,\eps)$, along with (a representation of)
        coefficients $c_{i,j}$ to a semilinear threshold tester such that, given $\Poi(k)$ samples from either $P$ or an arbitrary distribution $Q$ with $\ell_1$ distance $\geq\eps$ from $P$, the tester distinguishes these cases except with error probability $\leq e^{\Delta(P,k,\eps)}$. 
        \item For positive integer $s$, on the problem of testing the hypothesis $P^{sub(s)}$ with $\Poi(k\cdot s)$ samples, the same tester as above, with coefficients $c_{i,j}$, will have error probability $\leq e^{s\cdot \Delta(P,k,\eps)}$. 
        \item By contrast, for {\bf any} sequence of testers $T_s$ indexed by $s$---including arbitrary nonlinear testers---letting $\delta_s$ denote the failure probability of tester $s$ for testing the hypothesis $P^{sub(s)}$ with $\Poi(k\cdot s)$ samples to accuracy $\eps$, we have that \[\liminf_{s\rightarrow\infty} \frac{1}{s} \log\delta_s\geq \Delta(P,k,\eps)\]
    \end{itemize}
\end{theorem}

\subsection{Perspective}
\subsubsection{The choice of $\ell_1$ metric}
We briefly comment on the choice of the $\ell_1$ metric in defining the hypothesis testing problem: our tester is asked to reject the hypothesis $P$ on any distribution $Q$ such that $|P-Q|_1\geq\eps$. However many other distance metrics are possible, including the KL-divergence, Hellinger distance, or chi-squared divergence---all of which show up in related statistical contexts. We choose $\ell_1$ for this paper both because it is standard for work in the statistical property testing field for related estimation tasks. Further, and importantly, since $\ell_1$ distance is equivalent to total variation distance, it exactly captures how distinguishable a single sample from $P$ versus $Q$ is, thus making it a natural benchmark.

However, the optimization framework introduced in this paper is far more general than just the $\ell_1$ distance distribution hypothesis testing framework we work with here. Many aspects of our analysis extend unchanged to other metrics, and understanding the implications of this would presumably yield interesting follow-up papers. In short, we choose the $\ell_1$ metric but the approach we introduce here is not in any way wed to it.

\subsubsection{Comparison with~\cite{VV17}}
Our main result is a bit unusual in comparison with standard results about testing properties of distributions. Compare, for example, with the main theorem of~\cite{VV17} which describes a precise formula for the number of samples needed to run accurate hypothesis testing as a function of the distribution $P$, and the accuracy $\eps$ (that paper assumes $\delta=\Theta(1)$\,): the main result there characterizes the number of samples in terms of the $\ell_{\frac{2}{3}}$ norm of the hypothesis $P$. By contrast, our optimization framework is a rather opaque expression compared to the $\ell_{\frac{2}{3}}$ norm result of~\cite{VV17}.

We view our current approach as furthering and complementing the goals of~\cite{VV17}, for the following reasons. First, our results here are tighter, both in theory, and in practice. While~\cite{VV17} obtained performance that was tight to constant factors, the results here are $1+o(1)$ tight, and in simulations (see figures in Section 1) are essentially identical to performance of the ground truth optimal algorithm in situations where the ground truth can be numerically ascertained. Constant factors are often critical in data-efficiency contexts, and a focus of increasing recent interest. Second, while the current paper does not provide a clean formula for the performance of our new algorithm---see Equation~\ref{eq:relaxed} and below, which presumably have no closed-form solution---the formula from~\cite{VV17} is already shown to be constant-factor tight by that paper, so may be used here as a closed-form estimate of the performance of our algorithm as well. The fact that our optimization framework does not appear to have a closed-form solution is not a fault of our paper but a fact of nature: this expression is optimal yet this expression is not simple. This paper provides access to optimality via an algorithm that efficiently finds the testing coefficients and the objective value; and this seems like the best that can be hoped for.

\subsubsection{Comparison to~\cite{GP22}}
This recent paper focuses on the distribution hypothesis testing problem, restricted to the case of \emph{uniform} distributions. The paper shares our focus on $1+o(1)$-tight analysis. The main results are identifying an asymptotic regime where Pearson's chi-squared test performs optimally, and then modifying the chi-squared test---in a semilinear way, to instead compute the Huber function on each domain element---leading to $1+o(1)$-optimal performance across a fairly wide range of parameter space.

The main distinction between our paper and~\cite{GP22} is that we consider generic distributions instead of just uniform distributions. It is not at all clear how one could optimally combine uniformity testers to test more general distributions, although our optimization framework in some sense can be viewed as the prescription for this. We briefly point out that non-uniform distributions are (of course) more common than uniform distributions, so uniformity testing can be considered a special case, but whose nice mathematical structure does not always reflect real life. We briefly mention 3 examples where non-uniform hypotheses might arise. Pearson's chi-squared test is often used to compare two empirically accessed distributions, and in many cases, one of these distributions may be much more expensive to sample than the other; and in this case, one can call the inexpensive distribution the ``known" hypothesis $P$ as it is so cheap to obtain farther samples. Beyond empirically obtained distributions, there are also many cases in science where prior knowledge leads one to hypothesize an explicit non-uniform distribution. For example, back in 1928, R.A. Fisher investigated Mendelian genetics and its prediction of non-uniform allele frequencies~\cite{Fisher1928}. A much more modern example is the question of how to test whether a quantum computer is accurately sampling from the distribution of outcomes expected by the laws of quantum mechanics, and distinguishing this from the case that the quantum computer is sampling from an erroneous distribution~\cite{wang2021qdiff}.

As further comparison between our paper and results of~\cite{GP22}, see the figures in Section 1 which compare the performance of our tester with several analyzed in~\cite{GP22}. Because of the asymptotic setup of~\cite{GP22}, the parameters of the ``Huber tester" they recommend are never specified---since the specifics do not matter in the limit considered in that paper. But we are thus unable to compare our algorithm to any specific proposal from~\cite{GP22}.

At a higher level, the Huber function arises in~\cite{GP22} out of the dual requirements that the function is quadratic near the origin and linear away from it. But the Huber function is just one choice of many possible functions that fit these criteria equally well. Our paper shows that in fact the $\log\cosh$ function is the correct (i.e., optimal) form in general.

Finally, we provide perspective on two of the quirks of our analysis.

\subsubsection{Depoissonization}
While our main theorem is phrased in terms of a Poisson-distributed number of samples $\Poi(k)$ from a distribution, the more standard setting is to take some fixed number $k$ of samples. As discussed above, we move to the ``Poissonized" setting because here, the number of times that each domain element $j$ is sampled becomes completely independent of the other domain elements, which is crucial for both the upper and lower bound analysis. However, the natural conclusion of a ``Poissonized" analysis is ``depoissonization"---showing how to move these results back to the original framework, with exactly $k$ samples. Handling this in full would substantially extend the length of this paper, so we instead just outline the method here.

We could simply take the tester $\mathbf{c}$ produced by our optimization and use it with exactly $k$ samples, and its error probability would increase by at most a factor inversely proportional to the probability that $\Poi(k)$ equals $k$. However, this would not be expected to match any corresponding lower bounds of the flavor of Theorem~\ref{thm:main}.

Instead, we adapt Equation~\ref{eq:relaxed} while preserving as much of its meaning as possible. Each of the two  components of the outer max of Equation~\ref{eq:relaxed}---representing type 1 and type 2 error--- represents a Chernoff bound on the failure probability, and thus the exponential of either of these can be considered as an expectation over the process of taking $\Poi(k)$ samples from a distribution that is either $P$ or $Q$. We consider ``tilting" this expectation by multiplying the terms of the objective function by some exponential $\alpha^i$ so that the maximum contribution to either of the two terms comes from cases where exactly $k$ samples are taken, instead of some other value in the support of $\Poi(k)$. In this sense, we essentially add another constraint to the optimization, that we are looking for a solution ``where $\Poi(k)$ mimics the deterministic value $k$". After some simplification, it turns out that this new constraint can be equivalently reexpressed as a new constraint on the inner max, stipulating that the ``worst case $Q$" that we construct (described via histogram variable $\histq$) must have the same total probability mass (namely, 1) as $P$. For technical reasons, such a constraint must be left \emph{out} of the original Equation~\ref{eq:relaxed} in order for Theorem~\ref{thm:main} to hold, but to produce the best tester for the deterministic $k$-sample regime, we add this extremely natural constraint to Equation~\ref{eq:relaxed}. We point out that adding a constraint to this maximization can only \emph{decrease} the overall objective value; thus we expect the $c_{i,j}$ customized to the deterministic case to outperform the Poissonized $c_{i,j}$, even when the Poissonized tester is evaluated on $\Poi(k)$ samples. While we do not include the details, we expect \emph{all} parts of Theorem~\ref{thm:main} to extend to this ``depoissonized" tester, including that, for fixed $j$, the coefficients $c_{i,j}$ will have the same log cosh functional form as before. The main difference will be that the additional constraint will induce a new dual variable---which we may call $\beta$. Thus when we express the overall optimization problem as a 2-level optimization, as in Equation~\ref{eq:merged-split}, there will be a third variable $\beta$ in the outermost optimization, joining $\alpha$ and $u$.

In short, the technical machinery of this paper extends naturally to ``depoissonizing" the estimator of Theorem~\ref{thm:main}, with very little change or computational overhead.

\subsubsection{``Chernoff bounds are tight"}

The approach of this paper is predicated on the intuition that finding the algorithm with optimal Chernoff bound for the failure probability in our setting should be viewed as morally the same task as finding the optimal algorithm; this, combined with the versatility of manipulating Chernoff bounds enables the approach of our paper.

In statistics, the intuition that ``Chernoff bounds are tight" is called the ``large deviations principle". We briefly comment on 3 old and extremely respected results establishing this intuition.

The first is Cram\'{e}r's theorem~\cite{Cramer38}, which considers the common case for Chernoff bounds where we wish to derive a tail bound on the mean of a large number of i.i.d. copies of a real-valued distribution $X$. Namely, for some threshold $\gamma$, we wish to derive a bound on $\Pr_{x_1,\ldots,x_n\leftarrow X}[\frac{1}{n}\sum_i x_i\geq\gamma]$. The standard Chernoff bound on this (provided $\gamma\geq\mu(X)$\,) is $\min_{t\geq 0} \Exp_{x\leftarrow X}[e^{t(x-\gamma)}]^n$. Taking logarithms and normalizing by $\frac{1}{n}$, we have that the best Chernoff bound does not change with $n$. In the limit as $n\rightarrow\infty$, how good is this Chernoff bound? Cram\'{e}r's theorem actually says that the limit of $\frac{1}{n}$ times the log of the tail probability \emph{exactly equals} the corresponding expression from the Chernoff (upper) bound. This theorem is extremely general, requiring almost no assumptions on the distribution $X$ except that the relevant quantities in the equality claim are well defined. The cost of this generality is that the convergence rate is hard to control.

While Cram\'{e}r's theorem shows convergence in the \emph{exponent} of the tail probability, a much more precise result is the Bahadur-Rao theorem~\cite{BR60}, which actually shows multiplicative convergence in the probability itself, and not just in the exponent. The full formulation is complicated, but it essentially says that the true tail bound is not only better than the optimal Chernoff bound, but in fact \emph{quasilinear} in the optimal Chernoff bound. Intuitively, suppose we are trying to decide between two probabilistic processes $A,B$, and we have Chernoff bounds on their failure probabilities (in some sense) denoted respectively $a,b$; then even in regimes where the Chernoff bounds $a,b$ are \emph{not} accurate enough, since both $a,b$ are quasilinear in their true failure probabilities, we can correctly choose between $A,B$ by choosing the better of $a,b$, trusting on quasilinearity to preserve the \emph{ordering} of the failure probabilities, even when ``distorted" by Chernoff bounds. In short, ``because Chernoff bounds are quasilinear in the true tail bounds, \emph{optimizing} over Chernoff bounds is as good as optimizing over the true tail probabilities".

Finally, there is a long history of generalizing Cram\'{e}r's theorem to settings well beyond means of i.i.d.~variables. One of the most natural and general variants is the G\"{a}rtner-Ellis theorem---see Section~\ref{sec:gartner-ellis}. While Cram\'{e}r's theorem talks about the convergence of tail bounds on a random variable $Z_n$ equal to the mean of $n$ i.i.d. copies of a distribution, G\"{a}rtner-Ellis by contrast says that, for almost any sequence of random variables $Z_n$, we can ``pretend, in the style of Cram\'{e}r's theorem, that $Z_n$ is the mean of $n$ i.i.d. copies of a distribution", and so long as the moment generating function of $Z_n$ converges to something differentiable as $n\rightarrow\infty$, we can conclude convergence analogously to Cram\'{e}r's theorem.

In short, we leverage a long history of tools and intuition that ``Chernoff bounds are tight", and we hope this paper exposes the benefits of these intuitions to a wider audience.

\iffalse\subsection{Experimental results}
\paul{draw shape of $c_{i,j}$}
\fi
\subsection{Related work}

We direct the reader to several notable surveys about distribution testing~\cite{gs009} and property testing~\cite{Dana08,goldreich_2017}, along with the paper \cite{BFFKRW01} that initialized much of the work in this area.

A historically and practically important special case of our problem of hypothesis testing on distributions is the problem of \emph{uniformity testing}, where the hypothesis is a uniform distribution. There is a long line of work on this topic, including relatively recent papers~\cite{paninski08, DGPP18, GP22}. We highlight this last paper which gets $1+o(1)$-tight bounds on the sample complexity of uniformity testing in various natural asymptotic regimes. This paper suggests modifying the chi-squared statistic into the Huber statistic, which is a function combining quadratic and linear regions. This paper demonstrates the value of looking at a richer class of functional forms for testers, and emphasizes the importance of $1+o(1)$-tight analysis when proposing new statistical testers.

The paper~\cite{VV17} focuses on a very similar setting to the one of the current paper, of hypothesis testing on distributions.. Its main feature is ``instance-optimal" results, that, up to some constant factor, perform as well as is possible for the particular hypothesis distribution $P$. This work revealed that the difficulty of testing $P$ depends on $||P||_{\frac{2}{3}}$, the $\frac{2}{3}$ norm of the distribution, an interesting structural insight. The instance-optimal analysis, however, is at the expense of some loss of constant factors, and it is unclear how well these algorithms would perform in practice.

Some classic results in the area, including parts of \cite{VV17}, were extended to the ``tolerant testing" regime in~\cite{acharya2015optimal}. A tolerant tester is not only required to accept distributions that satisfy the given property, but also tolerate some corruption/error without rejecting. The paper~\cite{acharya2015optimal} showed several tolerant testing constructions based around a tester that rejects distributions with $\ell_1$ distance $\geq\eps$, while accepting distributions with chi-squared distance $\leq \eps'$. This result shows how, in some cases, tolerant testing can be achieved without paying a high cost in the other parameters of the problem.

%focusedidentity testing, in which we are given samples from an unknown distribution and a distribution, we want to tell whether those samples are sampled from the given distribution or some distribution that is $\eps$- far away from the given distribution in $\ell_1$ distance with constant confidence. And in this paper, they provide an instance optimal tester for that problem. The instance optimality means that for any instance, this algorithm needs at most constant times of the samples required by the optimal algorithm for this instance, which provides a quite robust guarantee of the performance.(\hongao{bad English})

\cite{Goldreich20} shows a general relation between hypothesis testing on distribution and the simpler uniformity testing problem, with a general constant-factor tight reduction.

The paper~\cite{DGKPP21} encompases several previously studied regimes of these problems, examining how the sample complexity of testing the independence and closeness of distributions depends on the desired failure probability $\delta$. Many previous works took $\delta$ to be a constant---implicitly dealing with other $\delta$ with the black-box amplification process of taking a majority vote over independent repetitions of the tester on fresh data; by contrast, this paper examines when it is possible to combine the data in more subtle ways to get better performance as a function of $\delta$.%The main differences between this paper and the previous works are the bound in this paper contains the probability of failure $\delta$ as a parameter, while the other papers assume $\delta$ is a constant, usually $\frac{2}{3}$, such that we can use that algorithm as a black-box and the amplification process to get the arbitrary probability of success, which will usually get a sub-optimal algorithm instead of the optimal one.

Many variants of these problems exist. In the ``closeness testing'' problem we are given sample access to \emph{two} distributions, $P,Q$ (instead of being given $P$ explicitly) and the aim is to distinguish when $P=Q$ from when $||P-Q||_1\geq\eps$~\cite{BFRSW00, ADJOP11, BFRSW13, CDVV14}.

%focusing on closeness testing. In this setting, we are getting samples from two distributions and trying to tell whether these two distributions are equal or $\epsilon$-far away from each other, instead of getting samples from one distribution and telling whether they are taken from some given distribution or from some distribution that is $\eps$-far away from the given one.

Generalizing in a different direction, the Generalized Uniformity Testing model seeks to test uniformity without knowing the support of the distribution in advance~\cite{BC17, DKS18}.

There are also several works~\cite{CFGM16, CRS15, ACK14} on distribution testing on conditional samples. In this model, the testing algorithm has a kind of ``query access" to the distribution, in that the tester can name a subset as an input, and then receives samples \emph{conditioned} on lying in the named subset. By choosing the subset properly, they can obtain much stronger results, achieving constant sample complexity rather than $O(\sqrt{n})$.

There are also many works~\cite{DKN15, DDSVV13, CDSS14} focusing on the identity testing problem for \emph{structured distributions}, where the distribution is guaranteed to have certain global structural properties, like being monotonic, or $k$-modal. In many cases, this structure can be leveraged to yield dramatically (and sometimes exponentially) better performance.

Recent attention on quantum computing has led to several papers posing variants of the property testing problem under the quantum setting~\cite{OW15, BOW19, BCL20}. The goal is to seek an algorithms that can distinguish whether the mixed state has some given property or is $\epsilon$-far away from any mixed state possessing the property, with some constant confidence. Analogously to the classical case, the optimality of the algorithm is measured by ``copy complexity'', asking how many copies of the same mixed state are required to conduct the test. This problem is closely related to identity testing or property testing problem in the classical setting, as every mixed state can be seen as a probability distribution on the support of pure states, and by taking a measurement on one mixed state, we can get a sample from that distribution; this builds a direct relationship between the copy complexity and the classical sample complexity.

\section{Preliminaries}\label{sec:prelim}
%(Part of the analysis of the Chernoff bounds of Neyman-Pearson assumes the 2-point structure of the optimum of Equation~\ref{eq:relaxed}; we need to prove this somewhere!)
\subsection{Notation}
Globally, we use $P$ to denote our hypothesis distribution, and we use $j$ to index the \emph{unique} probabilities of $P$---so that we can collectively deal with all identical domain elements at once; let $y_j$ denote the probability of each of the domain elements indexed by $j$, and let $\histp$ denote the number of such domain elements. While each element of $P$ that has probability $y_j$ is symmetric, and should be treated symmetrically by our tester, different $j$ should be treated differently. Thus when we consider an alternative probability distribution $Q$, we separately represent a histogram for each equivalence class $j$ of our hypothesis distribution $P$. Namely, let $\histq$ be the histogram of $Q$ on the domain elements that have probability exactly $y_j$ on $P$. When describing a Chernoff bound, we often use the notation ``$\chernoff$" to emphasize that the inequality comes from a Chernoff bound.

\iffalse
We define the following notations.
\begin{itemize}
    \item Let $y^p_j$ be the $j^\textrm{th}$ \emph{distinct} probability mass of a probability distribution $p$. Usually, we use $y_j$ without the superscript.
    \item Let $h^p_x$ be the number of domain elements of $p$ that has probability mass $x$. Note that $h^p_x$ is only positive when $x = y_j$ for some $j$. Collectively, $h^p$ is a ``histogram'' of $p$.
    \item Some abuse of notation: we define $h^{q|y_j}$ to be the histogram of $q$ on the domain elements that have probability exactly $y_j$ on $p$.
    %\item For some fixed $n$ samples, let $F_{i, j}^p$ be the number of domain elements which has probability $y_j$ in $p$, and we have seen this domain element exactly $i$ times in the samples.
    %\item For some parameter $\theta$ and samples $(X_1,X_2,\cdots,X_n)$, we define $\Pr_{\theta}[(X_1,X_2,\cdots,X_n)]$ as the probability that the sample $(X_1,X_2,\cdots,X_n)$ is drawn from the distribution with parameter $\theta$. In other words, it is the likelihood. 
\end{itemize}
\fi

\subsection{The optimization problem}

In this section we describe and derive the overall optimization problem of finding the semilinear testing coefficients $\mathbf{c}$ that have the best possible Chernoff bounds on their performance.

Consider taking $\Poi(k)$ samples from some distribution $Q$, and then, finding the best Chernoff bound on the probability that the test will mistakenly say ``yes". We let $F_{i,j}$ denote the \emph{number} of different domain elements within the equivalence class $j$ that have been seen exactly $i$ times in the sample ($F$ for ``fingerprint'', as in~\cite{BFRSW00}). Thus the semilinear tester will compute the quantity $\sum_{i,j} c_{i,j}F_{i,j}$ and return ``yes" if this is $<0$. Thus by standard Chernoff bounds we have

\begin{align*}
    \Pr_{F\leftarrow Q^{\Poi(k)}}[\sum_{i, j} c_{i, j} F_{i, j} < 0] &\chernoff \inf_{t \le 0} \prod_{j, x} \left(\sum_i e^{t c_{i, j}} \poii{kx}\right)^{\histq} \\
        &= \inf_{t \le 0} \exp\left(\sum_{j, x} \histq \log \sum_i e^{t c_{i, j}} \poii{kx}\right),\\
\end{align*}
and conversely, when we draw samples from the hypothesis distribution $P$, we compute the best Chernoff bound on the probability that the tester mistakenly computes a statistic $\geq 0$ and thus outputs ``no":
\begin{align*}
    \Pr_{F\leftarrow P^{\Poi(k)}}[\sum_{i, j} c_{i, j} F_{i, j} \ge 0] &\chernoff \inf_{t' \ge 0} \prod_{j} \left(\sum_i e^{t' c_{i, j}} \poii{ky_j}\right)^{\histp} \\
        &= \inf_{t' \ge 0} \exp\left(\sum_{j} \histp \log \sum_i e^{t' c_{i, j}} \poii{ky_j}\right).\\
\end{align*}

Therefore, the following is an upper bound of the log of the error probability $\delta$ obtained by semilinear testers:
\begin{align*}  
    \hspace{-2cm}\min_{\mathbf{c}} \max_{Q : |Q - P|_1 \ge \eps} \log err(Q, \mathbf{c}) &\chernoff  \min_{\mathbf{c}} \max_{Q : |Q - P|_1 \ge \eps} \max\left\{\inf_{t \le 0} \sum_{j, x} \histq \log \sum_i e^{t c_{i, j}} \poii{kx},  \inf_{t' \ge 0} \sum_{j} \histp \log \sum_i e^{t c_{i, j}} \poii{ky_j}\right\}
\end{align*}

The constraint that $Q$ is a distribution such that $|Q-P|_1\geq\eps$ is informally stated, and must be rephrased in terms of the explicit optimization variables $\histq$. We have 3 constraints on the histogram $\histq$: the $\geq\eps$ distance constraint, the fact that for each $j$, the total number of domain elements described by $\histq$ for that $j$ equals the corresponding quantity $\histp$, and the constraint that $\histq\geq 0$; for the sake of efficient optimization, we must \emph{relax} the problem by omitting any \emph{integrality} constraint. Namely, even though $\histq$ represents the \emph{number} of domain elements satisfying certain conditions, we do not enforce the restriction that it must have an integer value. (The question of whether this relaxation is ``safe" is resolved, eventually, by the matching upper bound of Theorem~\ref{thm:main}.) We thus have our main expression upper-bounding the overall failure probability of our testing problem:
\begin{align}
\label{eq:relaxed}
    \min_{\mathbf{c}} \max\left\{\max_{\substack{\hist{q|y_j}{}:\sum_j \sum_x \histq |x-y_j|\geq\eps\\
%    \sum_j\sum_x x\cdot \histq=1\\
    \forall j,\; \sum_x \histq=\histp\\
    \forall j,\forall x,\; \histq\geq 0}}
    \min_{t\leq 0} \sum_{x,j} \histq \log \sum_i e^{t c_{i,j}} \poii{kx}\;,\;\min_{t'\geq 0}\sum_j \histp \log \sum_i e^{t' c_{i,j}} \poii{ky_j}\right\}
\end{align}

%The main result we want to prove about this is: we use $\hist{q|y_j}{}$ at optimum to construct a distribution-of-distributions, specifically a ``coin flip distribution"; we have a clean formula for the log-likelihood of an observation from this coin flip distribution; Neyman-Pearson says that the best tester is in fact a threshold tester (with some unknown threshold) of the log likelihood ratio. However, we really need the ``coin flip distribution \emph{conditioned} on being $\epsilon$-far from $P$ ($\histp$)". There isn't a closed form for these likelihoods; but Neyman-Pearson still applies; and we can get Chernoff bounds on Neyman-Pearson's performance. Miraculously, these Chernoff bounds exactly match Equation~\ref{eq:relaxed} (perhaps with the order of the outer min and max swapped, as below). So therefore, (we haven't shown our tester's performance is tight but) we have shown that the Chernoff bounds on performance are optimal, if Chernoff bounds are tight.

%Minimax theorem says that, \emph{given} $\hist{q|y_j}{}$, the best Chernoff bound on performance of any tester is the one for $c_{i,j}$, namely, the result of Equation~\ref{eq:relaxed}.

\section{Upper bound}
\label{sec:upper}

As a reminder, Equation~\ref{eq:relaxed} describes the best error bound of any semilinear tester, and we have denoted this quantity as $\Delta(P, k, \eps)$. In this section, we analyze the structure of the coefficients $c_{i, j}$ corresponding to this best error bound, as well as showing other desirable properties of the optimal variable values of Equation~\ref{eq:relaxed}. Later, in Section~\ref{sec:lower}, we show how to use the optimum value of $\histq$ to construct a distribution-over-distributions that will form the basis of the lower bound analysis, thus yielding the main theorem.

For the sake of applying Sion's minimax theorem soon, we first \emph{lower bound} Equation~\ref{eq:relaxed}, by restricting the domain of the inner max to ``$=\eps$" instead of ``$\geq\eps$". Later on, with Lemma~\ref{lem:relaxed-bound-tight}, we actually show that this lower bound is tight, so we incur no loss here.

\begin{align}
\label{eq:relaxed-bound}
    \min_{\mathbf{c}} \max\left\{\max_{\substack{\hist{q|y_j}{}:\sum_j \sum_x \histq |x-y_j|=\eps\\
    %\sum_j\sum_x x\cdot \histq=1\\
    \forall j,\; \sum_x \histq=\histp\\
    \forall j,\forall x,\; \histq\geq 0}}
    \min_{t\leq 0} \sum_{x,j} \histq \log \sum_i e^{t c_{i,j}} \poii{k x}\;,\;\min_{t'\geq 0}\sum_j \histp \log \sum_i e^{t' c_{i,j}} \poii{k y_j}\right\}
\end{align}

We aim to simplify this optimization problem by reducing the number of nested layers of min and max; and towards this end, we aim to swap the order of the max and the min of the first term. We will utilize Sion's minimax theorem to achieve this. We first state the theorem for clarity.

\begin{lemma}[Sion's minimax theorem; \cite{Sion58}]
\label{lem:Sion}
    Let $X$ and $Y$ be convex spaces, one of which is compact. Let $f : X \times Y \to \mathbb{R}$ be a function such that for all $x \in X$, $y \in Y$,
    \begin{itemize}
        \item $f(x, \cdot)$ is upper semi-continuous and quasi-concave on $Y$,
        \item $f(\cdot, y)$ is lower semi-continuous and quasi-convex on $X$,
    \end{itemize}
    then 
    \[\sup_{y} \inf_{x} f(x, y) = \inf_{x} \sup_{y} f(x, y).\]
\end{lemma}

The function being optimized is a convex function of $t$ because $e^{t c_{i,j}}\poii{kx}$ is log-convex, and sums of log-convex functions are log-convex, and thus its logarithm is convex. The function being optimized is linear in the maximization variable $\histq$. The domains of optimization are all convex as they are defined by linear constraints. We use limits to argue that the continuous domain of $x$ can be discretized (in the limit) without changing the optimum. In the following lemma, we will show the compactness of the space of $\histq$.

\begin{lemma}\label{lem:compact}
    Focusing on some particular $j$, and considering histogram locations $x$ that are multiples of some fixed spacing $\alpha$, we claim that the set of histograms satisfying the relaxed constraints $\sum_x h_x |x-y_j|\leq \eps$, and $\sum_x h_x=\histp$ and $\forall x,\, h_x\geq 0$ is compact when considered as a subset of the set of sequences, in the $\ell_1$ norm.
\end{lemma}
\begin{proof}
We reparameterize $h_x$ to index by nonnegative integers $i$, so that $h_i$ denotes what we previously referred to as $h_{\alpha i}$.

The constraints on $h$, reexpressed, now read: $\sum_{i=0}^\infty h_i |\alpha i-y_j|\leq \eps$, and $\sum_{i=0}^\infty h_i=\histp$ and $\forall i\geq 0,\, h_i\geq 0$. We show that this is a compact set by equivalently showing that this set is closed, bounded, and equismall at infinity\cite{treves2016topological}. %https://en.wikipedia.org/wiki/Sequence_space#Properties_of_%E2%84%931_spaces
The set is clearly closed; it is bounded since we explicitly have that the sum of the entries of $h$ equals the fixed parameter $\histp$. Equismall at infinity means that for every $\beta>0$, there exists an in integer $i_\beta$ such that $\sum_{i=i_\beta}^\infty |h_i|\leq \beta$ for all $h$. Explicitly, we take $i_\beta$ such that $\alpha i_\beta-y_j\geq\frac{\eps}{\beta}$. Thus for all $i\geq i_\beta$, we have that $|\alpha i-y_j|\geq\frac{\eps}{\beta}$. Thus suppose, for the sake of contradiction, that $\sum_{i=i_\beta}^\infty |h_i|> \beta$; multiplying these last two inequalities (and using $h_i\geq 0$) yields that $\sum_{i=i_\beta}^\infty h_i |\alpha i-y_j|> \beta\frac{\eps}{\beta}=\eps$, contradicting the first constraint. Thus these constraints describe a compact set in the $\ell_1$ topology.
\end{proof}

Thus, we can apply Lemma~\ref{lem:compact} to each $j$ separately, and conclude that the constraints on the histogram in Equation~\ref{eq:relaxed-bound} describe a set that lies in the direct product (across the finite set of $j$) of compact sets, and thus is itself compact. Therefore, we can apply Sion's minimax theorem(Lemma~\ref{lem:Sion}).

Thus Equation~\ref{eq:relaxed-bound} equals

\begin{align}
\label{eq:relaxed-bound2}
    \min_{\mathbf{c}} \max\left\{
    \min_{t\leq 0} \max_{\substack{\hist{q|y_j}{}:\sum_j \sum_x \histq |x-y_j|=\eps\\
    %\sum_j\sum_x x\cdot \histq=1\\
    \forall j,\; \sum_x \histq=\histp\\
    \forall j,\forall x,\; \histq\geq 0}}\sum_{x,j} \histq \log \sum_i e^{t c_{i,j}} \poii{k x}\;,\;\min_{t'\geq 0}\sum_j \histp \log \sum_i e^{t' c_{i,j}} \poii{k y_j}\right\}
\end{align}

%So thus Equation~\ref{eq:relaxed} equals

%\begin{align}
%\label{eq:relaxed-bound2}
%    \min_{\mathbf{c}} \max\left\{
%    \min_{t\leq 0} \max_{\substack{\hist{q|y_j}{}:\sum_j \sum_x \histq |x-y_j|\geq \eps\\
%    \sum_j\sum_x x\cdot \histq=1\\
%    \forall j,\; \sum_x \histq=\histp\\
%    \forall j,\forall x,\; \histq\geq 0}}\sum_{x,j} \histq \log \sum_i e^{t c_{i,j}} \poii{kx}\;,\;\min_{t'\geq 0}\sum_j \histp \log \sum_i e^{t' c_{i,j}} \poii{ky_j}\right\}
%\end{align}

We can pull both min's outside of the brackets, and also outside of the max, and hence merge them with the outer min to yield:

\begin{align}
   \min_{\mathbf{c},t\leq 0,t'\geq 0} \max\left\{
    \max_{\substack{\hist{q|y_j}{}:\sum_j \sum_x \histq |x-y_j|=\eps\\
   %\sum_j\sum_x x\cdot \histq=1\\
   \forall j,\; \sum_x \histq=\histp\\
   \forall j,\forall x,\; \histq\geq 0}}\sum_{x,j} \histq \log \sum_i e^{t c_{i,j}} \poii{k x}\;,\;\sum_j \histp \log \sum_i e^{t' c_{i,j}} \poii{k y_j}\right\}
\end{align}

The inner max is now a linear program over $\hist{q|y_j}{}$, so we take its dual. Let $\alpha$ be the (scalar) dual variable for the first constraint; let $\boldsymbol{\gamma}$ be the dual variable for the second constraint, a vector with an entry for each $j$.

\begin{align}
    \min_{\mathbf{c},t\leq 0,t'\geq 0} \max\left\{
     \min_{\substack{\alpha,\boldsymbol{\gamma}:\\\forall j,\forall x\geq 0,\; \alpha|x-y_j|+\gamma_j\geq \log \sum_i e^{t c_{i,j}} \poii{k x}}}\eps\alpha +\sum_j \gamma_j \histp\;,\;\sum_j \histp \log \sum_i e^{t' c_{i,j}} \poii{k y_j}\right\}
\end{align}

We pull the min outside of the brackets, outside of the max, and merge it with the outer min:

\begin{align}
\label{eq:relaxed-dual}
    \min_{\substack{\mathbf{c},t\leq 0,t'\geq 0,\alpha,\boldsymbol{\gamma}:\\\forall j,\forall x\geq 0,\; \alpha|x-y_j|+\gamma_j\geq \log \sum_i e^{t c_{i,j}} \poii{k x}}} \max\left\{
     \eps\alpha +\sum_j \gamma_j \histp\;,\;\sum_j \histp \log \sum_i e^{t' c_{i,j}} \poii{k y_j}\right\}
\end{align}

Summarizing the above, we have:
\begin{lemma}\label{lem:dual-manipulation}
Equation~\ref{eq:relaxed-bound} equals
    Equation~\ref{eq:relaxed-dual}.
\end{lemma}

Next, to understand how Equation~\ref{eq:relaxed-bound} relates to Equation~\ref{eq:relaxed}, we need to understand the role of the dual and the dual variable $\alpha$.
\begin{lemma}
\label{lem:alpha-non-zero}
The optimum of Equation~\ref{eq:relaxed-dual} is $\leq 0$. If the optimum is 0, then there is an optimal solution with $\alpha=0$; and if the optimum of Equation~\ref{eq:relaxed-dual} is $<0$ then $\alpha< 0$.
\end{lemma}
\begin{proof}
If the value of Equation~\ref{eq:relaxed-dual} equals 0, then, consider setting $\alpha=\gamma_j=t=t'=0$, in which case all the exponentials will equal 1, and, since $\sum_i \poii{kx}=1$ for any $x$, both logarithm expression equal 0, and the constraints are satisfied and the overall expression has value 0. Thus we have constructed an optimal solution satisfying the first claims of the lemma. Further, the optimal value can never be $>0$ because the solution above gives an objective value of 0, no matter what the the inputs $\eps,\histp,y$ are.

For the last part of the lemma, suppose for the sake of contradiction $\alpha\geq 0$. Consider setting $x$ from the constraints below the min to equal $y_j$.

Then, the first term in the max is at least \[\sum_j \gamma_j \histp\geq \sum_j \histp\log \sum_i e^{t c_{i,j}} \poii{ky_j}\]

The right hand side is convex as a function of $t$ (since $e^{t c_{i,j}}$ is log-linear and thus log-convex, and the sum of log-convex functions is log-convex; hence its logarithm is convex, and the entire expression is a sum of convex functions and hence convex). Further, its value at $t=0$ is 0. Thus, by convexity, its values at $t\leq 0$ and $t'\geq 0$ cannot \emph{both} be negative. Therefore the maximum of these is $\geq 0$, contradicting our assumption that the objective is negative.
\end{proof}

\begin{lemma}
\label{lem:relaxed-bound-tight}
   Equation~\ref{eq:relaxed} equals Equation~\ref{eq:relaxed-bound}.%, if either is $<0$.
\end{lemma}

\begin{proof}
Recall that $\alpha$ in Equation~\ref{eq:relaxed-dual} is the dual variable corresponding to the constraint of Equation~\ref{eq:relaxed-bound2} that the $\ell_1$ distance from the hypothesis equals $\eps$. Thus the fact that Equation~\ref{eq:relaxed-dual} has an optimal solution with $\alpha\leq 0$, from Lemma~\ref{lem:alpha-non-zero}, means that the optimum of Equation~\ref{eq:relaxed-bound2} would be unchanged if we changed this $=\eps$ constraint to $\geq\eps$. Equation~\ref{eq:relaxed-bound} with a $\geq\eps$ constraint is at most Equation~\ref{eq:relaxed-bound2} with a $\geq\eps$ constraint (trivially, from the ``weak" direction of minimax), which from the above equals Equations~\ref{eq:relaxed-bound} and~\ref{eq:relaxed-bound2} as written. Conversely, Equation~\ref{eq:relaxed-bound} with a $\geq\eps$ bound is at least Equation~\ref{eq:relaxed-bound} as written, since enlarging the domain of a maximum can only increase its value. Thus Equation~\ref{eq:relaxed-bound} must equal Equation~\ref{eq:relaxed-bound} with a $\geq\eps$ bound, namely Equation~\ref{eq:relaxed}.
\end{proof}

Combining Lemmas~\ref{lem:dual-manipulation} and~\ref{lem:relaxed-bound-tight} yields that Equation~\ref{eq:relaxed} equals Equation~\ref{eq:relaxed-dual}.
% \hongao{Lemma~\ref{lem:relaxed-bound-tight} implies that we can change the constraint $\sum_j \sum_x \histq |x-y_j|\geq\eps$ to $\sum_j \sum_x \histq |x-y_j|= \eps$, without changing the optimal value of the problem. And then lemma~\ref{lem:alpha-non-zero} implies that we can assume $\alpha < 0, t<0, t'>0$.}

% That means the problem we want to talk about here would be the following, for any fixed $\alpha,t,t'$ \hongao{in the above scope}:
% \begin{align}
% \label{eq:relaxed-dual-fixed-alpha}
% \min_{\substack{\mathbf{c},\boldsymbol{\gamma}:\\\forall j,\forall x,\; \alpha|x-y_j|+\gamma_j\geq \log \sum_i e^{t c_{i,j}} \poii{kx}}} \max\left\{
%      \eps\alpha +\sum_j \gamma_j \histp\;,\;\sum_j \histp \log \sum_i e^{t' c_{i,j}} \poii{ky_j}\right\}
% \end{align}

Looking closely at the inner max in Equation~\ref{eq:relaxed-dual}, we note that the first term is the dual of the type II error probability i.e, the probability of the tester accepting a sample from $Q$, while the second term is the type I error probability, i.e. the probability of the tester rejecting a sample from $P$. Intuitively, we expect both error probabilities to be equal in the optimum. We formalize this intuition with the following lemma.

\begin{lemma}
\label{lem:two-sides-equal}
There exists an optimum for Equation~\ref{eq:relaxed-dual} such that the two terms in the max are equal.
\end{lemma}
\begin{proof}
Letting $\alpha=\gamma_j=t=t'=0$ yields a feasible point for Equation~\ref{eq:relaxed-dual} with both terms in the max equal to 0.

The only remaining case is when the value of the optimum is $<0$. Observe that in this scenario, $t'>0$ (otherwise, $t=0$ and thus the second term of the max is equal to 0, and hence the overall value of the objective function is at least 0). Consider an optimal solution $\mathbf{c},\boldsymbol{\gamma}$. We separately consider the cases where the first term of the max is bigger, or smaller.

Suppose the first term is bigger. Then since $t'>0$, find a nonzero hypothesis value $\histp$, and increase the corresponding $c_{i,j}$ until the second term of the max equals the first. The variable $c_{i,j}$ only occurs in one other place, namely the constraint of the min. Because $t\leq 0$, increasing $c_{i,j}$ will decrease the right-hand side of the inequality constraint, and thus will not violate it. Hence we have constructed a feasible solution with the same objective value but which now satisfies the claim of the lemma.

For the other case, suppose the first term is smaller. We increase $\boldsymbol{\gamma}$ until the first term equals the second term. The only other places $\boldsymbol{\gamma}$ appears are on the left-hand side of the inequality constraints below the min; increasing $\alpha$ can only increase this left-hand side. Hence we have constructed a feasible solution with the same objective value but which now satisfies the claim of the lemma.
\end{proof}

We next split Equation~\ref{eq:relaxed-dual} into a component for each $j$, expressing the overall optimization as a ``nested'' optimization, choosing $\alpha,t,t'$ on the outside, and coefficients $\kappa_{i,j}$ inside. Here we use $\kappa_{i,j}$ in place of $c_{i,j}$ and $u$ to replace $t$ and $t'$. The following lemma will show that this change does not have any effect on the results of our optimization problem. And then the further lemmas will rely on this nested optimization to derive clean structural properties of the optimum.
%##\paul{somewhere we want to explain the relationship between $\kappa_{i,j}$ and $c_{i,j}$; not sure exactly what we want to say; $\kappa_{i,j}$ is shift invariant, so it's slightly different from $c_{i,j}$. also, as motivation/context, can introduce this somewhat-unmotivated equation by saying it's a shift-invariant formulation, with the added feature of disentangling the dependence across $j$. actually, I just changed it so it also reflects the scale-invariance with respect to $t,t$, and inversely scaling $c_{i,j}$}

\begin{lemma}\label{lem:merged}
Equation~\ref{eq:relaxed-dual} equals

\begin{equation}\label{eq:merged-split}\hspace{-2.5cm}\min_{\substack{\alpha\leq 0\\u\in[0,1]}}\left(\eps\alpha(1-u)+\sum_j\min_{\kappa_{i,j}}\left(u \histp\left(\log \sum_i e^{(1-u) \kappa_{i, j}} \poii{ky_j}\right)+(1-u)\histp\max_{x\geq 0}\left(-\alpha|x-y_j|+\log \sum_i e^{-u \kappa_{i, j}} \poii{kx}\right)\right)\right)\end{equation}
\end{lemma}
\begin{proof}
%\paul{need the proof to deal with the case where $t'-t=0$??? also, maybe replace $t,t'$ with $u$ here to simplify what's below?}
Given any optimal point of Equation~\ref{eq:relaxed-dual}, specified by $c_{i,j},t\leq 0,t'\geq 0,\alpha,\gamma_j$, we consider it in the context of Equation~\ref{eq:merged-split} (noting, by Lemma~\ref{lem:alpha-non-zero} that $\alpha\leq 0$ at optimum), and setting $\kappa_{i,j}=(t'-t) c_{i,j}$ and $u=-\frac{t}{t'-t}$ (with $u$ set arbitrarily if $t=t'=0$). Substituting in the relation $t c_{i,j}=-u\kappa_{i,j}$ into the constraint in Equation~\ref{eq:relaxed-dual}, we have \begin{equation}
    \label{eq:gamma-inequality}
\gamma_j\geq \max_{x\geq 0}\left(-\alpha|x-y_j|+\log \sum_i e^{-u \kappa_{i, j}} \poii{kx}\right)
\end{equation}

And thus, since $(1-u)\kappa_{i,j}=t'c_{i,j}$, the convex combination of the two terms in the max of the objective of Equation~\ref{eq:relaxed-dual}, with weights $1-u,u$, after substituting Equation~\ref{eq:gamma-inequality} for each $\gamma_j$ in the first term, is greater than or equal to the objective function of Equation~\ref{eq:merged-split}. Thus by the triangle inequality, this max in Equation~\ref{eq:relaxed-dual}---which is its objective function value at optimum---is at least the feasible value we have found for Equation~\ref{eq:merged-split}. Thus Equation~\ref{eq:merged-split} is less than or equal to Equation~\ref{eq:relaxed-dual}.

Conversely, given any feasible point of Equation~\ref{eq:merged-split}, we construct the corresponding feasible point of Equation~\ref{eq:relaxed-dual}, by setting $t=-u$, $t'=1-u$ (so that $t'-t=1$), $c_{i,j}=\kappa_{i,j}+s$ for a shift $s$ to be determined later, and setting \[\gamma_j = \max_{x\geq 0}\left(-\alpha|x-y_j|+\log \sum_i e^{t c_{i, j}} \poii{kx}\right)\]
We note that $\gamma_j$ depends linearly on $s$ with slope $t$, and thus the first term of the max in the objective of Equation~\ref{eq:relaxed-dual} depends linearly on $s$ with slope $t$ (since $\sum_j \histp=1$). Similarly, the second term in the max in the objective of Equation~\ref{eq:relaxed-dual} depends linearly on $s$ with slope $t'$. Thus we pick $s$ which makes these two terms equal.

Thus the objective value of Equation~\ref{eq:relaxed-dual} here equals both of the two terms of its max (because we just chose $s$ to make these terms equal); and in particular equals the convex combination of them with weights $1-u=t',u=-t$, which is thus exactly the objective value of Equation~\ref{eq:merged-split}, since the contributions of $s$ (whatever they are) exactly cancel with the weights. Thus Equation~\ref{eq:relaxed-dual} is less than or equal to Equation~\ref{eq:merged-split}.

Combining both parts yields the desired equality.
\end{proof}

We now show the uniqueness of the optimal solutions for $\kappa_{i, j}$ in Equation~\ref{eq:merged-split}, up to an irrelevant $j$-dependent additive shift.
\begin{lemma}
\label{lem:unique}
The inner minimization in Equation~\ref{eq:merged-split}, if $u\in(0,1)$, has a unique solution for $\kappa_{\cdot,j}$, for each $j$, up to an additive shift.
\end{lemma}
\begin{proof}
Consider two different optima, $\kappa_{i,j}$ and $\kappa'_{i,j}$. Since Equation~\ref{eq:merged-split} is convex in $\kappa$, and in particular, is the sum of convex terms, we have that each term must in particular be linear at any convex combination of $\kappa_{i,j}$ and $\kappa'_{i,j}$. In particular, $\log \sum_i e^{(1-u) \kappa_{i, j}} \poii{ky_j}$ must be linear, meaning that $\sum_i e^{(1-u) \kappa_{i, j}} \poii{ky_j}$ must be exponential as we interpolate between $\kappa_{i,j}$ and $\kappa'_{i,j}$. But the only way for the sum of several exponentials to equal an exponential is if all the exponentials have the same base; thus $\kappa_{i,j}=\kappa'_{i,j}+s$ for some additive shift, as claimed.    
\end{proof}

This shift-invariant structure of $\kappa_{i, j}$'s allows us an extra degree of freedom, as we can choose the shifting factor to further simplify the equation. In the following lemma, we reexpress the inner minimization in Equation~\ref{eq:merged-split} by taking its dual and then simplifying while utilizing the shift-invariant property; this also allows us to find the crucial and insightful closed form for $\kappa_{\cdot, j}$.

\begin{lemma}\label{lem:closed-form-kij}
The inner minimization in Equation~\ref{eq:merged-split}, if $u\in(0,1)$, can be expressed for any $j$ as 
\begin{equation}\label{eq:max-formulation}\max_{a_x\geq 0} \histp\log\frac{\sum_i \left(\sum_x a_x \poii{k x}\right)^{1-u}\poii{k y_j}^u}{\left(\sum_x a_x e^{\alpha|x-y_j|}\right)^{1-u}}\end{equation}
where the optimal variable values of Equation~\ref{eq:merged-split} can be expressed in terms of the optimal variables of Equation~\ref{eq:max-formulation} as
\begin{equation}\label{eq:cij-expression}\kappa_{i,j}=\log\sum_x a_x\frac{ \poii{kx}}{ \poii{ky_j}}\end{equation}
\end{lemma}
\begin{proof}
We first point out that because Equation~\ref{eq:merged-split} is invariant to shifts in $\kappa_{i,j}$ for each $j$, there is some shift of the optimal $\kappa_{i,j}$ that makes the inner max equal to 0. Thus the optimum is unchanged if we restrict the inner maximum to have value $\leq 0$ and remove this term from the objective. Exponentiating both the objective and the new constraints---and for now omitting the scaling factor $u \histp$ from the objective function, since it merely scales the optimum---yields

\begin{align}
&\min_{\kappa_{\cdot,j}}\left( \sum_i e^{(1-u) \kappa_{i,j}} \poii{ky_j}\right) \notag\\
&\textrm{ such that } \forall x,\; 1\geq e^{-\alpha|x-y_j|} \left(\sum_i e^{-u \kappa_{i,j}} \poii{kx}\right) \label{eq:split-with-constraint}
\end{align}
The dual of this convex optimization is then:
\begin{align}
\label{eq:split-Lagrange-dual}
\max_{\lambda_x\geq 0} \min_{\kappa_{i,j}} \left( \sum_i e^{(1-u) \kappa_{i,j}} \poii{ky_j}\right)+\sum_x \lambda_x \left(-1+e^{-\alpha|x-y_j|} \left(\sum_i e^{-u \kappa_{i,j}} \poii{kx}\right)\right)
\end{align}

To solve the inner minimization, we compute the derivative with respect to a single $\kappa_{i,j}$: \begin{align*}&(1-u)e^{(1-u)\kappa_{i,j}}\poii{k y_j}+\sum_x \lambda_x e^{-\alpha|x-y_j|} (-u)e^{-u\kappa_{i,j}}\poii{k x}\\
=& e^{(1-u)\kappa_{i,j}} (1-u)\poii{k y_j}+e^{-u\kappa_{i,j}}\sum_x \lambda_x e^{-\alpha|x-y_j|} (-u)\poii{k x}\end{align*}

Setting this to 0 yields $\kappa_{i,j}=\log \frac{u}{1-u}\sum_x \lambda_x e^{-\alpha|x-y_j|} \frac{\poii{k x}}{\poii{k y_j}}$

Plugging this into Equation~\ref{eq:split-Lagrange-dual} (to evaluate the dual problem), the objective function is:

\begin{align*}\hspace{3cm}&\hspace{-3cm}\left( \sum_i \left(\frac{u}{1-u}\sum_x \lambda_x e^{-\alpha|x-y_j|} \frac{\poii{k x}}{\poii{k y_j}}\right)^{1-u} \poii{ky_j}\right)\\&+\sum_x \lambda_x \left(-1+e^{-\alpha|x-y_j|} \left(\sum_i \left(\frac{u}{1-u}\sum_x \lambda_x e^{-\alpha|x-y_j|} \frac{\poii{k x}}{\poii{k y_j}}\right)^{-u} \poii{kx}\right)\right)\end{align*}
which simplifies to

\[\sum_i \frac{1}{u}\left(\frac{u}{1-u}\sum_x \lambda_x e^{-\alpha|x-y_j|} \poii{k x}\right)^{1-u} \poii{ky_j}^u-\sum_x\lambda_x\]

Since this expression must be maximized in terms of $\lambda_x$, it must in particular be maximized with respect to scalings of $\lambda_x$. Replacing $\lambda_x$ with $e^z\lambda_x$ for some parameter $z$, we calculate that the maximum over $z$ of this expression equals
\[\left(\frac{\sum_i \left(\sum_x \lambda_x e^{-\alpha|x-y_j|} \poii{k x}\right)^{1-u} \poii{ky_j}^u}{\left(\sum_x \lambda_x\right)^{1-u}}\right)^{\frac{1}{u}}\]

For convenience, we reexpress $\lambda_x$ as $a_x e^{\alpha|x-y_j|}$, where for each $x$, we have $a_x$ is nonzero if and only if $\lambda_x$ is nonzero. Taking the logarithm of this, and multiplying back by the factor $u \histp$ that we dropped earlier yields that the inner minimization of the $j^\textrm{th}$ term in Equation~\ref{eq:merged-split} equals the maximization

\[\max_{a_x\geq 0} \histp\log\frac{\sum_i \left(\sum_x a_x \poii{k x}\right)^{1-u}\poii{k y_j}^u}{\left(\sum_x a_x e^{\alpha|x-y_j|}\right)^{1-u}}\]
as claimed. Finally, we point out that, defining $\kappa_{i,j}$ in terms of $a_x$ yields $\kappa_{i,j}=\log\frac{u}{1-u}\sum_x a_x\frac{\poii{k x}}{\poii{k y_j}}$; since $\kappa_{i,j}$ in Equation~\ref{eq:merged-split} is shift-invariant, and $a_x$ is scale-invariant (since we have rescaled by the optimal scaling factor $e^z$), we can drop the $\frac{u}{1-u}$ factor without affecting anything, yielding Equation~\ref{eq:cij-expression}.
\end{proof}

%\paul{there's some magic in the relation between $\lambda_x$ and $a_x$ - this is essentially the gap between the deterministic worst case distribution $p_x$ and the worst case coin flip distribution $q_x$ I think; the relation between these is meant to be a bit mysterious. - Actually, with the new dual, things are way less mysterious?}

The above expression for the optimal values $\kappa_{i,j}$ can be viewed as representing $\kappa_{i,j}$, for any fixed $j$, as the sum of exponentials in $i$:
\begin{equation}\label{eq:cij-expression2}
\kappa_{i,j}=\log\sum_x a_x\frac{ \poii{kx}}{\poii{ky_j}}=\log\sum_x a_x e^{k(y_j-x)}(x/y_j)^i\end{equation}

We will now show a crucial property of the hardest distribution $Q$ with respect to $\kappa_{i, j}$: each distinct probability mass $y_j$ in $P$ is replaced by \emph{exactly} two probability masses, which we denote as $x_{1,j}$ and $x_{2,j}$. This property is known to be the worst case for $\eps$-far uniformity testing for both the TV and collision tester \cite{DGPP18}. In the context of Equation~\ref{eq:cij-expression}, for any fixed $j$, we would expect that exactly two $a_x$'s are non-zero, which is formally shown below in Lemma~\ref{lem:two-point}. En route to this lemma, we will show an additional property of $\boldsymbol{\kappa}$ as well as a technical lemma to aid with the final analysis of the two-point structure.

\begin{lemma}\label{lem:c-convex}
The optimal value of $\kappa_{i,j}$ given by Equation~\ref{eq:cij-expression} is convex as a function of $i$ (where Equation~\ref{eq:cij-expression} applies, by Lemma~\ref{lem:closed-form-kij}, if $u\in(0,1)$\,).
\end{lemma}
\begin{proof}
We reexpress Equation~\ref{eq:cij-expression} as Equation~\ref{eq:cij-expression2}, and point out that exponential functions are log-convex, and sums of log-convex functions are log-convex, yielding that $\kappa_{i,j}$ is a convex function of $i$.    
\end{proof}

\begin{lemma}\label{lem:log-concave-mixture}
Suppose $g:\mathbb{Z}^+\rightarrow \mathbb{R}$ is log-concave. Then $\sum_{k=0}^\infty g(k) poi(\lambda,k)$ is a log-concave function of $\lambda$ for $\lambda\geq 0$; and if $g$ is strictly log-concave somewhere, then this function is strictly log-concave for all $\lambda>0$.
\end{lemma}
\begin{proof}
We instead work with the function $f(\lambda)=e^{\lambda}\sum_{k=0}^\infty g(k) poi(\lambda,k)=\sum_{k=0}^\infty g(k)\lambda^{k}\frac{1}{k!}$, which differs from the desired function only in the $e^{\lambda}$ factor, and thus is log-concave exactly when the original function is.

We have $f'(\lambda)=\sum_{k=0}^\infty g(k+1)\lambda^{k}\frac{1}{k!}$ and $f''(\lambda)=\sum_{k=0}^\infty g(k+2)\lambda^{k}\frac{1}{k!}$. We note that we can extend the sum down to, say, $k=-1$ under the convention that $(-1)!=\infty$.

The condition for log-concavity is: $f'(\lambda)^2-f''(\lambda)f(\lambda)\geq 0$. Writing this out as a double sum: \[f'(\lambda)^2-f''(\lambda)f(\lambda)=\sum_{j,k=0}^\infty  (g(j+1)g(k+1)-g(j+2)g(k))\lambda^{j+k}\frac{1}{j!k!}\]

Replacing $j$ by $j-1$ yields the identical sum

\[\sum_{j=1}^\infty\sum_{k=0}^\infty  (g(j)g(k+1)-g(j+1)g(k))\lambda^{j+k-1}\frac{1}{(j-1)!k!}\]

Adding this sum to the corresponding sum with $j,k$ swapped (where we symmetrize the domain of the sum to the set of integers $j,k\geq 0$ where $j+k\geq 1$, which does not change the summation):

\[\sum_{\substack{j,k\geq 0\\j+k\geq 1}} \lambda^{j+k-1} (g(j)g(k+1)-g(j+1)g(k))\left(\frac{1}{(j-1)!k!}-\frac{1}{(k-1)!j!}\right)\]

We now point out that this sum is nonnegative term-by-term: the expression $g(j)g(k+1)-g(j+1)g(k)$ is nonnegative for log-concave $g$ when $j\geq k$, and symmetrically, nonpositive when $j\leq k$; the term $\left(\frac{1}{(j-1)!k!}-\frac{1}{(k-1)!j!}\right)$ is also nonnegative when $j\geq k$ and nonpositive when $j\leq k$; hence their product is always nonnegative, meaning the overall log-concavity condition is the sum of nonnegative terms, and is thus nonnegative, as desired.

To show strict log-concavity, we point out that the term $\left(\frac{1}{(j-1)!k!}-\frac{1}{(k-1)!j!}\right)$ is strictly positive whenever $j>k$; and thus if there is a location $\ell\geq 1$ (namely in the interior of $\mathbb{Z}^+$) such that $g(\ell+1)^2>g(\ell)g(\ell+2)$ then, letting $j=\ell$ and $k=\ell-1$ yields $g(j)g(k+1)-g(j+1)g(k)>0$ and hence the overall log-concavity expression is strictly positive, multiplied by some power of $\lambda>0$.
\end{proof}

We are now ready to prove the two-point structure of $q$.

\begin{lemma}
    \label{lem:two-point}
    For any fixed $\alpha<0,u\in(0,1)$, and fixed $j$, the inner minimization of Equation~\ref{eq:merged-split} will have have a solution $\kappa_{i,j}$ that, when expressed via Equation~\ref{eq:cij-expression}, has exactly two $x$ such that $a_x \neq 0$, which we denote as $x_{1,j}<x_{2,j}$, and which satisfy $x_{1,j}<y_j<x_{2,j}$. We normalize so that we express $a_{x_{1,j}}=q_j$ and $a_{x_{2,j}}=1-q_j$, for some $q_j\in(0,1)$---taking advantage of the fact that Equation~\ref{eq:merged-split} is invariant to shifting $\kappa_{\cdot,j}$, and thus we can scale $a_x$ arbitrarily.% Further, $x_{1,j}<y_j<x_{2,j}$, and the inner maximum in Equation~\ref{eq:merged-split} is tight exactly at $x_{1,j},x_{2,j}$, and the $x$ derivative of this maximum will be 0 at these points, except, possibly, at the boundary of the domain, in the case that $x=x_{1,j}=0$.
\end{lemma}

\begin{proof}
If $\kappa_{i,j}$ is exactly linear as a function of $i$ (for our fixed $j$), namely $\kappa_{i,j}=a +b i$, then, expressing $\kappa_{i,j}$ in the form of Equation~\ref{eq:cij-expression2}, there must be a single nonzero $a_x$.

% If $\kappa_{i,j}$ is exactly linear as a function of $i$ (for our fixed $j$), namely $\kappa_{i,j}=a +b i$, then we have a closed form expression for the sum in the $max_x$ term:
%  \[\sum_i e^{t \kappa_{i,j}}\poii{kx}=\sum_i e^{t (a+b i)} \frac{e^{-kx} (kx)^i}{i!}=e^{ta+(e^{tb}-1)kx}\]

% Namely, its log is a linear function of $x$. In the $\max_x$ term this is added to an absolute value function (times a positive [???] $-\alpha$). 
% Thus if the $\max_x$ term is finite (which it must be if the expression is optimized), we must have either that the max is attained at a single point, or that $y_j=0$. By assumption this latter cannot happen [???trivial].

Otherwise, since by Lemma~\ref{lem:c-convex}, $\kappa_{i,j}$ is a convex function of $i$, then $\kappa_{i,j}$ must be \emph{strictly} convex (since we already covered the case where it is linear). Thus by Lemma~\ref{lem:log-concave-mixture}, we have that the expression $\log \sum_i e^{-u \kappa_{i, j}} \poii{kx}$ in the inner max of Equation~\ref{eq:merged-split} is a strictly concave function of $x$; and thus must intersect with each branch of the absolute value function at most once, namely at most twice in total.

Thus, considering the dual form of Equation~\ref{eq:merged-split} analyzed in the proof of Lemma~\ref{lem:closed-form-kij}, by complementary slackness the dual variable $\lambda_x$ must be nonzero in at most 2 locations. And by definition of $a_x$, it is nonzero if and only if the corresponding $\lambda_x$ is.

Thus in all cases, the maximization $\max_x -\alpha|x-y_j|+ \log \sum_i e^{-u \kappa_{i, j}} \poii{kx}$ is tight for at most 2 points, with at most one intersection point per branch of $\alpha|x-y_j|$. Thus we call $x_{1,j}$ the (possible) intersection point below $y_j$ and call $x_{2,j}$ the (possible) intersection point above $y_j$. 

If there are no intersections, than $a_x=0$ everywhere, and $\kappa_{i,j}=-\infty$ and the expression $\log\sum_i e^{-u \kappa_{i,j}}\poii{kx}$ is infinite (since $t<0$) and thus the optimization of Equation~\ref{eq:split-with-constraint} is infeasible, so this cannot happen. If there is one intersection, then, there is at most one nonzero $a_x$ (by complementary slackness), and thus $\kappa_{i,j}$ is linear in $i$ by Equation~\ref{eq:cij-expression2}. If $\kappa_{i,j}$ is linear in $i$ then $\sum_i e^{-u \kappa_{i, j}} \poii{kx}$ is linear in $x$. Now, if a linear function is bounded by the constraint $\alpha|x-y_j|+\beta$ and intersects it once, the intersection point must be at $x=0$. Thus the only nonzero $a_x$ can be $a_0$. This leads to $\kappa_{i,j}=\log a_0 e^{k(y_j-0)}(0/y_j)^i$, which when $i\geq 1$ is $\log 0=-\infty$. As above, the expression $\log\sum_i e^{-u \kappa_{i,j}}\poii{kx}$ is thus infinite (since $u>0$) and thus the constraint of Equation~\ref{eq:split-with-constraint} is infeasible, so this cannot happen.

Thus it must be that we have two intersection points, $x_{1,j}$ and $x_{2,j}$. Namely, $\alpha|x-y_j|+\gamma_j=\sum_i e^{-u \kappa_{i, j}} \poii{kx}$ at both $x=x_{1,j}$ and $x=x_{2,j}$. Neither $x_{1,j}$ nor $x_{2,j}$ can equal $y_j$ since $\sum_i e^{-u \kappa_{i, j}} \poii{kx}$ is smooth, and thus the maximum of this smooth function plus the ``v-shaped'' function $-\alpha|x-y_j|$ cannot occur at the corner of the ``v'' (since $\alpha<0$).% Further, since $-\alpha|x-y_j|+\sum_i e^{-u \kappa_{i, j}} \poii{kx}$ is differentiable everywhere except $x=y_j$, it must be differentiable at $x_{1,j},x_{2,j}$, where it attains its maximum, and these derivatives must be 0, except possibly at the boundary of the domain, if $x_{1,j}=0$.
\end{proof}
%\paul{do we need tangency conditions any longer, in above lemma?}
% Thus, by Equation~\ref{eq:cij-expression}, we have that, for some $q_j,x_{1,j},x_{2,j}$, up to additive constant, we have
% \begin{equation}
% \label{eq:c-ij}
%     c_{i,j}=\frac{1}{t'-t}\log \frac{q_j \poii{kx_{1,j}}+(1-q_j)\poii{kx_{2,j}}}{\poii{ky_j}}
% \end{equation}
% (Note that we could also express this as $\frac{1}{t'-t}$ times the log of hyperbolic cosine, shifted and scaled, plus a linear function.)
% \paul{remember for further down that $x_{1,j}$ might equal 0, and in this case we will \emph{not} have tangency here}

We next show, via a relatively straightforward calculation, that we can explicitly find a negative objective function value for our optimization problem, hence implying that the global optimum value is $<0$. We will then use this to rule out certain pathological behavior that could occur when certain variables equal 0.

\begin{lemma}
\label{lem:negative-opt}
The optimum value of Equation~\ref{eq:merged-split} is $<0$, attained in the limit as $\alpha\rightarrow 0$ from the negative side, and setting $u = \frac{1}{2}$, $\tau = \left(-\alpha\frac{y_j^2}{u k^2}\right)^{\frac{1}{3}}$, $x_{1,j}=y_j-\tau$, $x=y_j+\tau$, and $\kappa_{i, j} = \log \frac{\poii{kx_{1,j}} + \poii{kx_{2,j}}}{2 \poii{ky_j}}$.
\end{lemma}

The final result of this section, and of the upper-bound portion of our analysis, puts together the pieces we have shown so far to summarize the properties of the optimum of Equation~\ref{eq:merged-split}---and hence also Equation~\ref{eq:relaxed}. We use the structural properties of the upper bound to reexpress our optimal objective function value, via Equation~\ref{eq:opt-reformulated}, in a form that will show up crucially in our lower bound, in Lemma~\ref{lem:lower-matching-with-3-factors}.

\begin{lemma}\label{lem:upper-final}
When Equation~\ref{eq:merged-split} is optimized for $\alpha,u,\kappa_{i,j}$, then $\alpha<0$, $u\in(0,1)$, and (the optimal) $\kappa_{i,j}$ can be expressed via Lemma~\ref{lem:two-point} as $\kappa_{i,j}=\log \frac{ q_j\poii{k x_{1,j}}+(1-q_j)x_{2,j}}{ \poii{ky_j}}$, and its value equals \begin{equation}\label{eq:opt-reformulated}\eps\alpha(1-u)+\sum_j \histp\log \frac{\sum_i (q_j \poii{kx_{1,j}}+(1-q_j)\poii{kx_{2,j}})^{1-u}\poii{ky_j}^{u}}{(q_j e^{\alpha|x_{1,j}-y_j|}+(1-q_j) e^{\alpha|x_{2,j}-y_j|})^{1-u}}\end{equation}
which by Lemmas~\ref{lem:dual-manipulation},~\ref{lem:relaxed-bound-tight}, and~\ref{lem:merged} equals Equation~\ref{eq:relaxed}.
Further, for these $\alpha,u,q_j,x_{1,j},x_{2,j}$, the derivative of Equation~\ref{eq:opt-reformulated} with respect to $\alpha$ or $u$ or any $q_j$ is 0.
\end{lemma}

\begin{proof}
We first show that we can apply Lemma~\ref{lem:two-point}. From Lemma~\ref{lem:negative-opt}, the objective function at optimum must be negative. We now point out that, since Equation~\ref{eq:merged-split} equals Equation~\ref{eq:relaxed-dual} (by Lemma~\ref{lem:merged}), then Lemma~\ref{lem:alpha-non-zero} yields that $\alpha<0$. Next, we point out that $u\neq 1$ because if $u=1$ then all the terms in Equation~\ref{eq:merged-split} vanish, leaving an objective of 0; but the objective is actually negative.
Finally, we point out that $u\neq 0$, since otherwise, if $u=0$, the inner minimization of Equation~\ref{eq:merged-split} contains the term $\max_{x\geq 0} -\alpha|x-y_j|$ which is infinite since $\alpha<0$, and hence contradicts the fact that the optimum is negative. Thus the conditions of Lemma~\ref{lem:two-point} apply, namely $\alpha<0$ and $u\in(0,1)$.

We can thus apply Lemma~\ref{lem:closed-form-kij} for each $j$, where $\kappa_{i,j}$ is reexpressed using the 2-point form of Lemma~\ref{lem:two-point}, and we use the other part of Lemma~\ref{lem:closed-form-kij} to say that the inner minimization of Equation~\ref{eq:merged-split} equals the inner maximization below, to conclude that Equation~\ref{eq:merged-split} equals %Equation~\ref{eq:opt-reformulated}.

\begin{equation}\label{eq:min-max}\min_{\alpha<0,u\in(0,1)} \left(\eps\alpha(1-u)+\sum_j \max_{q_j,x_{1,j},x_{2,j}} \histp\log \frac{\sum_i (q_j \poii{kx_{1,j}}+(1-q_j)\poii{kx_{2,j}})^{1-u}\poii{ky_j}^{u}}{(q_j e^{\alpha|x_{1,j}-y_j|}+(1-q_j) e^{\alpha|x_{2,j}-y_j|})^{1-u}}\right)\end{equation}

By Lemmas~\ref{lem:dual-manipulation},~\ref{lem:relaxed-bound-tight}, and~\ref{lem:merged} we have that Equation~\ref{eq:merged-split} equals Equation~\ref{eq:relaxed}, so thus Equation~\ref{eq:opt-reformulated} equals Equation~\ref{eq:relaxed}, as desired.

Since the function being optimized in Equation~\ref{eq:min-max} is smooth, and the inner maximization has a unique solution (by Lemma~\ref{lem:unique}), and the optimum is attained for $\alpha,u,q_j$ in the interior of their domains (for $q_j$ this is from Lemma~\ref{lem:two-point}), thus the derivatives of Equation~\ref{eq:opt-reformulated} with respect to any of $\alpha,u,q_j$ are all 0, as desired.
\end{proof}

\section{Lower bound}
\label{sec:lower}

Given a hypothesis distribution specified by $P$, and a number of samples $k$, we derive a particular distribution ``$Q$'' that is \emph{far} from the hypothesis $P$, yet \emph{hard to distinguish} from it, and we derive our lower bounds from this. If $Q$ is $\eps$-far from $P$ and no $k$-sample tester can distinguish $P$ from $Q$ with success probability $\geq 1-\delta$, then, a fortiori, no $k$-sample tester can distinguish $P$ from the entire set of distributions $\geq\eps$ distance from $P$, with success probability $\geq 1-\delta$.

Instead of providing a single distribution $Q$, we instead provide a \emph{distribution over distributions} $\mathcal{Q}_{[\eps,\eps']}$, each of whose members is itself $\geq\eps$-far from $P$, and such that no $\Poi(k)$-sample tester can distinguish $P$ from a random distribution from $\mathcal{Q}_{[\eps,\eps']}$ with expected success probability $\geq 1-\delta$.

Explicitly, for each domain element, we will flip a (weighted) coin between 2 probabilities; we call this a ``coin flip distribution"; we then condition on the $\ell_1$ distance of this distribution from our hypothesis $P$ being in the interval $[\eps,\eps']$, and call $\mathcal{Q}_{[\eps,\eps']}$ a ``conditional coin flip distribution''. We choose $\eps'$ later so that it converges to $\eps$.%\paul{where???}

In this section we will find it convenient to regard $j$ as representing a single domain element, instead of as an equivalence class of identical domain elements; thus $y_j$ refers to the probability of domain element $j$ in the hypothesis; and we avoid referring to $\histp$.

\begin{definition}
Let $\mathcal{Q}$ be the distribution-over-distributions defined according to the following process. (We note that, because we are in the Poissonized setting, the total probability mass of a distribution output by $\mathcal{Q}$ need not be exactly 1.) From Lemma~\ref{lem:upper-final}, we take optimal coefficients $q_j,x_{1,j},x_{2,j}$ for each $j$. %(where $j$ is associated with $\histp$ domain elements each of whom have identical hypothesized probability $y_j$). 
For each domain element, with associated $j$, we flip a weighted coin, and with probability $q_j$ set this domain element to have probability $x_{1,j}$ and otherwise, have probability $x_{2,j}$.

We further define $\mathcal{Q}_{[\eps,\eps']}$ to be a \emph{conditional} distribution, where we sample a distribution $q$ from $\mathcal{Q}$ but proceed to return $q$ \emph{conditioned} on whether its $\ell_1$ distance from $P$ is $\in [\eps,\eps']$.

For a histogram $h$, recording for each domain element, the number of times it was sampled, we let $\mathcal{Q}(h),\mathcal{Q}_{[\eps,\eps']}(h)$ denote the probability of $h$ appearing as this histogram from $\Poi(k)$ samples from $q$ drawn respectively from $\mathcal{Q}$ or $\mathcal{Q}_{[\eps,\eps']}$. We will also sometimes consider a realization $r$ of the coin flip process in $\mathcal{Q}$, where $r_j$ denotes whether the $j^\textrm{th}$ coin is heads or tails. Let $\mathcal{Q}(h,r)$ denote the joint probability that the coin flips realized outcome $r$, and then that the histogram sampled from this was $h$. For the sake of convenience, we define the joint probability of observing a histogram $h$ and where the coin flip realization has $\ell_1$ distance $\eps$ from the hypothesis: $\mathcal{Q}(h,[\eps,\eps'])=\sum_{r:\sum_j |x_{r_j,j}-y_j|\in[\eps,\eps']} \mathcal{Q}(h,r)$; and we define the probability that the coin flip realization has $\ell_1$ distance $\in[\eps,\eps']$ from the hypothesis, regardless of the histogram sampling process: $\mathcal{Q}([\eps,\eps'])=\sum_h \mathcal{Q}(h,[\eps,\eps'])$. Using this notation, we can express the law of conditional probability: $\mathcal{Q}_{[\eps,\eps']}(h)=\frac{\mathcal{Q}(h,[\eps,\eps'])}{\mathcal{Q}([\eps,\eps'])}$.

%We will refer to sampling a \emph{distribution} from $\mathcal{Q}$ or $\mathcal{Q}_{|\geq\eps}$; and we will also refer to the process of sampling a \emph{histogram}, where first we sample a distribution $q\leftarrow\mathcal{Q}$, and then in a second sampling process, we take $\Poi(k)$ samples from $q$ and return their histogram.
\qed
\end{definition}

For the sake of completeness, we state and prove the standard Neyman-Pearson Lemma, showing that, to distinguish 2 hypotheses, the optimal tester (over \emph{all} possible testers!) is a log-likelihood threshold tester.

\begin{lemma}[Neyman-Pearson]\label{lem:neyman-pearson}
The tester for distinguishing $P$ from $\mathcal{Q}_{[\eps,\eps']}$ such that the max of type-1 and type-2 error is minimized, is defined by a log-likelihood threshold $\gamma$, and a tie-breaking probability $\lambda$, where the tester says ``no" if $\log\frac{\mathcal{Q}_{[\eps,\eps']}(h)}{P(h)}> \gamma$, ``yes'' if the log-likelihood is $<\gamma$, and flips a $\lambda$-biased coin if the log-likelihood equals $\gamma$. Equivalently, the tester says ``no" according to a $\lambda$-weighted coin flip between the output of whether $\log\frac{\mathcal{Q}_{[\eps,\eps']}(h)}{P(h)}\geq \gamma$ and the strict version $\log\frac{\mathcal{Q}_{[\eps,\eps']}(h)}{P(h)}> \gamma$.
\end{lemma}%\paul{haven't read/edited below proof yet??? (e.g., change $\geq\eps$ to $\in[\eps,\eps']$\,)}
\begin{proof}
    First, we will define the log-likelihood threshold tester as follows:
    \begin{equation}
        \Gamma(h) = 
        \begin{cases}
        1 & \textrm{if $\log\frac{\mathcal{Q}_{[\eps,\eps']}(h)}{P(h)} > \gamma$}\\
        1 & \textrm{with probability $p_{\gamma}$ if $\log\frac{\mathcal{Q}_{[\eps,\eps']]}(h)}{P(h)} = \gamma$}\\
        0 & \textrm{otherwise}
        \end{cases}
    \end{equation}
    We will then consider the best log-likelihood threshold tester
    \[\arg\min_\Gamma \max\left\{\Pr_{h\leftarrow P}\left[\Gamma(h) = 1\right]\,,\,\Pr_{h\leftarrow \mathcal{Q}_{[\eps,\eps']}}\left[\Gamma(h) = 0\right]\right\}\]
    We claim that this minimum is achieved when $\Pr_{h\leftarrow P}\left[\Gamma(h) = 1\right]=\Pr_{h\leftarrow \mathcal{Q}_{[\eps,\eps']}}\left[\Gamma(h) = 0\right]$. Suppose for the sake of contradiction that when it takes its minimum, we have $\Pr_{h\leftarrow P}\left[\Gamma(h) = 1\right]<\Pr_{h\leftarrow \mathcal{Q}_{[\eps,\eps']}}\left[\Gamma(h) = 0\right]$. As we can choose $p_{\gamma}$ and $\gamma$ as we wish, so we can increase $\Pr_{h\leftarrow P}\left[\Gamma(h) = 1\right]$ by increasing $p_{\gamma}$ and $\gamma$, and by changing this, $\Pr_{h\leftarrow \mathcal{Q}_{[\eps,\eps']}}\left[\Gamma(h) = 0\right]$ will not increase. Thus, we can get a better tester unless $\Pr_{h\leftarrow P}\left[\Gamma(h) = 1\right]=\Pr_{h\leftarrow \mathcal{Q}_{[\eps,\eps']}}\left[\Gamma(h) = 0\right]$, which contradicts the hypothesis that this is the best tester. A corresponding proof applies for the case $\Pr_{h\leftarrow P}\left[\Gamma(h) = 1\right]>\Pr_{h\leftarrow \mathcal{Q}_{[\eps,\eps']}}\left[\Gamma(h) = 0\right]$. Thus, the optimal log-likelihood threshold tester $\Gamma(\cdot)$, where the max of type-1 error and type-2 error is minimized, has these two terms are equal to each other.
    
    We set the tester $\Gamma(h)$ as the tester above and $\gamma$ as the threshold in the tester above, we will show that this is the optimal tester that will minimize the max of type-1 error and type-2 error among all testers $\Gamma^*(h)$. In order to do that, notice the following:
    \begin{align*}
        & \Exp_{h\leftarrow \mathcal{Q}_{[\eps,\eps']}}\left[\Gamma(h)\right]- \Exp_{h\leftarrow \mathcal{Q}_{[\eps,\eps']}}\left[\Gamma^*(h)\right] - e^{\gamma} \left(\Exp_{h\leftarrow P}\left[\Gamma(h)\right]- \Exp_{h\leftarrow P}\left[\Gamma^*(h)\right]\right)\\
        & = \sum_h \left(\Gamma(h) - \Gamma^*(h)\right)\left(Q_{[\eps,\eps']}(h)-e^{\gamma}P(h)\right) \\
        & = \sum_{h\in L} \left(\Gamma(h) - \Gamma^*(h)\right)\left(Q_{[\eps,\eps']}(h) - e^{\gamma}P(h) \right) +\sum_{h\in M} \left(\Gamma(h) - \Gamma^*(h)\right)\left(Q_{[\eps,\eps']}(h) - e^{\gamma} P(h)\right)\\
        & = \sum_{h\in L} \left(1 - \Gamma^*(h)\right)\left(Q_{[\eps,\eps']}(h) - e^{\gamma}P(h) \right) +\sum_{h\in M} \left( - \Gamma^*(h)\right)\left(Q_{[\eps,\eps']}(h) - e^{\gamma} P(h)\right)
    \end{align*}
    Here $L = \{h: Q_{[\eps,\eps']}(h) > e^{\gamma}P(h)\}$, $M = \{h: Q_{[\eps,\eps']}(h) < e^{\gamma}P(h)\}$. So we have $\Gamma(h) = 1$, when $h\in L$ and $\Gamma(h) = 0$, when $h\in M$. And notice that $\Gamma^*(h) \in [0,1]$, so both components of the equation above are non-negative, which means the equation above is non-negative. And we notice that:
    \begin{align*}
        & \Exp_{h\leftarrow \mathcal{Q}_{[\eps,\eps']}}\left[\Gamma(h)\right]- \Exp_{h\leftarrow \mathcal{Q}_{[\eps,\eps']}}\left[\Gamma^*(h)\right] - e^{\gamma} \left(\Exp_{h\leftarrow P}\left[\Gamma(h)\right]- \Exp_{h\leftarrow P}\left[\Gamma^*(h)\right]\right)\\
        & = \Pr_{h\leftarrow \mathcal{Q}_{[\eps,\eps']}}\left[\Gamma(h) = 1\right]- \Pr_{h\leftarrow \mathcal{Q}_{[\eps,\eps']}}\left[\Gamma^*(h) = 1\right] - e^{\gamma} \left(\Pr_{h\leftarrow P}\left[\Gamma(h) = 1\right]- \Pr_{h\leftarrow P}\left[\Gamma^*(h) = 1\right]\right)\\
        & = \Pr_{h\leftarrow \mathcal{Q}_{[\eps,\eps']}}\left[\Gamma^*(h) = 0\right] -\Pr_{h\leftarrow \mathcal{Q}_{[\eps,\eps']}}\left[\Gamma(h) = 0\right] - e^{\gamma} \left(\Pr_{h\leftarrow P}\left[\Gamma(h) = 1\right]- \Pr_{h\leftarrow P}\left[\Gamma^*(h) = 1\right]\right)
    \end{align*}
    And by the definition of $\Gamma(h)$, we know that $\Pr_{h\leftarrow P}\left[\Gamma(h) = 1\right] = \Pr_{h\leftarrow \mathcal{Q}_{[\eps,\eps']}}\left[\Gamma(h) = 0\right]$. Then suppose $\Gamma^*$ is a better tester, which means $\Pr_{h\leftarrow \mathcal{Q}_{[\eps,\eps']}}\left[\Gamma^*(h) = 0\right] < \Pr_{h\leftarrow \mathcal{Q}_{[\eps,\eps']}}\left[\Gamma(h) = 0\right]$ and $\Pr_{h\leftarrow P}\left[\Gamma^*(h) = 1\right] < \Pr_{h\leftarrow P}\left[\Gamma(h) = 1\right]$. However, this will make the equation above negative, which contradicts the fact that the equation is non-negative for any possible tester $\Gamma^*$. So there is no such a tester $\Gamma^*(h)$ with a better max of type-1 and type-2 error than $\Gamma(h)$. That finishes the proof.
\end{proof}

Given that the optimal tester for distinguishing 2 hypotheses is a log-likelihood threshold tester, we now look to understand Chernoff bounds on the performance of such testers.

\begin{lemma}\label{lem:pq-chernoff}
For any two distributions $P,Q$, we have that the log-likelihood tester with threshold $\gamma$ has failure probability bounded by the Chernoff bound
\[\Pr_{x\leftarrow P}\left[\log\frac{Q(x)}{P(x)}\geq \gamma\right]\chernoff \min_{u\leq 1} \sum_x P(x)^u Q(x)^{1-u} e^{-\gamma (1-u)}\]
\end{lemma}
\begin{proof}
We simply take the Chernoff bound of the distribution that takes value $\log\frac{Q(x)}{P(x)}$ with probability $P(x)$; for Chernoff parameter $t$, quick simplification gives $u=1-t$; hence $t\geq 0$ translates to the bound $u\leq 1$.
\end{proof}

Plugging in $P,\mathcal{Q}_{[\eps,\eps']}$ to Lemma~\ref{lem:pq-chernoff} with some threshold $\gamma$, and also, symmetrically, plugging in $\mathcal{Q}_{[\eps,\eps']},P$ with the corresponding threshold $-\gamma$ yields
\begin{corollary}\label{lem:first-chernoff}
\[\min_\gamma \max\left\{\Pr_{h\leftarrow P}\left[\log\frac{\mathcal{Q}_{[\eps,\eps']}(h)}{P(h)}\geq \gamma\right]\,,\,\Pr_{h\leftarrow \mathcal{Q}_{[\eps,\eps']}}\left[\log\frac{\mathcal{Q}_{[\eps,\eps']}(h)}{P(h)}\leq \gamma\right]\right\}\chernoff \min_{u\in [0,1]} \sum_h P(h)^u \mathcal{Q}_{[\eps,\eps']}(h)^{1-u} \] 
\end{corollary}
\begin{proof}
Explicitly plugging into Lemma~\ref{lem:pq-chernoff} (where for the second invocation, the expression inside the probability is easily seen to be the negative of what is in Lemma~\ref{lem:pq-chernoff}, and hence equivalent) the left hand side is bounded by \begin{equation}\label{eq:pq-chernoff}\min_\gamma\max\left\{\min_{u\leq 1} \sum_h P(h)^u \mathcal{Q}_{[\eps,\eps']}(h)^{1-u} e^{-\gamma (1-u)}\,,\, \min_{u\geq 0} \sum_h P(h)^u \mathcal{Q}_{[\eps,\eps']}(h)^{1-u} e^{\gamma u}\right\}\end{equation}
If $\gamma=0$, then both components of the max are equal; and the min is attained at $u\in[0,1]$ since $\sum_h P(h)^u \mathcal{Q}_{[\eps,\eps']}(h)^{1-u}$ equals 1 at $u=0$ and $u=1$ and is convex;  and thus the overall expression equals $\min_{u\in [0,1]} \sum_h P(h)^u \mathcal{Q}_{[\eps,\eps']}(h)^{1-u}$ as claimed.

Otherwise, without loss of generality, we consider the case $\gamma>0$.

If 0 is not an optimal $u$ for $\min_{u\geq 0} \sum_h P(h)^u \mathcal{Q}_{[\eps,\eps']}(h)^{1-u} e^{\gamma u}$  then the $e^{\gamma u}$ term is strictly greater than 1, and thus the min is greater than $\min_{u\in [0,1]} \sum_h  P(h)^u \mathcal{Q}_{[\eps,\eps']}(h)^{1-u}$, implying that $\gamma$ is not actually the optimal choice of $\gamma$, a contradiction.

Otherwise, 0 is an optimal $u$, meaning that our expression is $\sum_h \mathcal{Q}_{[\eps,\eps']}(h)=1$. And if 1 is the optimal value, then since this value is attained at both $u=0$ and $u=1$ and the expression $a^u b^{1-u}$ is convex in $u$ for any $a,b\geq 0$, we have that the Equation~\ref{eq:pq-chernoff} attains its minimum of 1, for $\gamma=0$, equaling $\min_{u\in [0,1]} \sum_h P(h)^u \mathcal{Q}_{[\eps,\eps']}(h)^{1-u}$ as desired.
\end{proof}

We point out that the Chernoff bounds for a $\leq$ tail are the same as those for a $<$ tail. %\paul{And thus, intuitively, a ``Chernoff bounds are tight" result should apply equally for strict and non-strict tails.}

In light of Lemma~\ref{lem:neyman-pearson},  Corollary~\ref{lem:first-chernoff} is bounding the performance of the best distinguisher for $P$ versus $\mathcal{Q}_{[\eps,\eps']}$. %\paul{essentially bounding from both directions, given a ``Chernoff bounds are tight" result}

And the best distinguisher of $P$ from all distributions $\geq\eps$ far from $P$ must in particular distinguish $P$ from $\mathcal{Q}_{[\eps,\eps']}$ .

From Lemma~\ref{lem:neyman-pearson}, the error of the optimal tester is at least
    \[\min_\gamma \max\left\{\Pr_{h\leftarrow P}\left[\log\frac{\mathcal{Q}_{[\eps,\eps']}(h)}{P(h)}>\gamma\right]\,,\,\Pr_{h\leftarrow \mathcal{Q}_{[\eps,\eps']}}\left[\log\frac{\mathcal{Q}_{[\eps,\eps']}(h)}{P(h)}< \gamma\right]\right\}\]

If the optimal $\gamma$ is $\geq 0$ then the max of these two expressions is at least the first one, which is at least $\Pr_{h\leftarrow P}\left[\log\frac{\mathcal{Q}_{[\eps,\eps']}(h)}{P(h)}>0\right]$; and the best Chernoff bound on this is $\min_{u\in[0,1]} \sum_h P(h)^u \mathcal{Q}_{[\eps,\eps']}(h)^{1-u}$.

On the other hand, if the optimal $\gamma$ is $\leq 0$, then the error of the optimal tester is at least the second term $\Pr_{h\leftarrow \mathcal{Q}_{[\eps,\eps']}}\left[\log\frac{\mathcal{Q}_{[\eps,\eps']}(h)}{P(h)}< \gamma\right]$, which is at least $\Pr_{h\leftarrow \mathcal{Q}_{[\eps,\eps']}}\left[\log\frac{\mathcal{Q}_{[\eps,\eps']}(h)}{P(h)}< 0\right]$; and the best Chernoff bound on this is the expression from above $\min_{u\in[0,1]} \sum_h P(h)^u \mathcal{Q}_{[\eps,\eps']}(h)^{1-u}$. 

\begin{definition}
We thus define $ch_1$ to be the smaller of the ratio between $\Pr_{h\leftarrow P}\left[\log\frac{\mathcal{Q}_{[\eps,\eps']}(h)}{P(h)}>0\right]$ or $\Pr_{h\leftarrow \mathcal{Q}_{[\eps,\eps']}}\left[\log\frac{\mathcal{Q}_{[\eps,\eps']}(h)}{P(h)}< 0\right]$ and the common Chernoff bound $\min_{u\in[0,1]} \sum_h P(h)^u \mathcal{Q}_{[\eps,\eps']}(h)^{1-u}$.

\qed

%And define $ch_1$ to be the ratio between this and the best Chernoff bound on it, for the $\gamma$ that minimizes the above failure probability: \[\max\left\{\min_{u\leq 1} \sum_h P(h)^u \mathcal{Q}_{|\geq\eps}(h)^{1-u} e^{-\gamma (1-u)}\,,\, \min_{u\geq 0} \sum_h P(h)^u \mathcal{Q}_{|\geq\eps}(h)^{1-u} e^{\gamma u}\right\}\]
\end{definition}

%\paul{I think it can be shown that if $\gamma>0$ only the second term matters, and vice versa.}

Namely, the failure probability of the best tester is at least \[ch_1\cdot\min_{u\in [0,1]} \sum_h P(h)^u \mathcal{Q}_{[\eps,\eps']}(h)^{1-u}\]

%\paul{first equation from below lemma is copied from a lemma further below; this notation doesn't match up with what's above though!}

%\paul{still thinking the following through???}

Towards understanding this bound and reexpressing it via a \emph{second} layer of Chernoff bounds, we define the following ``coin flip" random variables.

\begin{definition}\label{def:pi}
    Let $q_{2,j}=1-q_{1,j}=q_j$ for convenience. Let $h_j$ be the number of times domain element $j$ is sampled. Define, for each $j\in\{1,
    \ldots,n\}$ the independent random coin-flip variable $X_j$ so that for each $r_j\in \{1,2\}$---denoting the outcome of the $j^\textrm{th}$ coin flip---the variable $X_j$ takes value $|x_{r_j,j}-y_j|$ with probability $\pi_{r_j,j}$ defined to be \[\pi_{r_j,j}=\sum_{h_j} \frac{poi(k\,y_j,h_j)^u}{(q_1 poi(k\,x_{1,j},h_j)+q_2 poi(k\,x_{2,j},h_j))^u}q_{r_j,j}\,poi(k\,x_{r_j,j},h_j)\]
    where we emphasize that $X_j$ is an unnormalized distribution, in that its probabilities may not sum to 1.
\end{definition}

\begin{lemma}
    For $u\in [0,1]$, and using the probabilities $\pi$ from Definition~\ref{def:pi}, we have
    \begin{equation}
    \label{eq:lower-lower}
    \sum_h P(h)^u \mathcal{Q}_{[\eps,\eps']}(h)^{1-u}\geq \frac{1}{\mathcal{Q}([\eps,\eps'])} \sum_{r:\sum_j |x_{r_j,j}-y_j|\in[\eps,\eps']}  \prod_j \pi_{r_j,j}=\frac{1}{\mathcal{Q}([\eps,\eps'])}\Pr_{x_1\leftarrow X_1,\ldots, x_n\leftarrow X_n}\left[\sum_j x_j\in[\eps,\eps']\right]%\sum_{h_j} \frac{poi(y_j,h_j)^t}{(q_1 poi(x_{1,j},h_j)+q_2 poi(x_{2,j},h_j))^t}q_{r_j,j}\,poi(x_{r_j,j},h_j)
        %\sum_h P(h)^{u} \left(\frac{Q(h,\geq\eps)}{Q(\geq\eps)}\right)^{1-u}\geq \sum_h \frac{P(h)^{u}}{Q(h)^{u} }\frac{Q(h,\geq\eps)}{Q(\geq\eps)^{1-u}}
    \end{equation}

The best Chernoff bound for the right hand side is \begin{equation}\label{eq:lower-lower-chernoff}\frac{1}{\mathcal{Q}([\eps,\eps'])}\min_{s\geq 0} e^{-s\eps}\prod_j \left( e^{s|x_{1,j}-y_j|}\pi_{1,j}+ e^{s|x_{2,j}-y_j|}\pi_{2,j}\right) \end{equation}
\end{lemma}

\begin{proof}
%Let $h_j$ be the number of times domain element $j$ is sampled, and $r_j$ the outcome of the $j^{\textrm{th}}$ coin flip.
We have, since the joint probability $\mathcal{Q}(h,[\eps,\eps'])$ is at most the probability of the marginal $\mathcal{Q}(h)$, that
\begin{align}\sum_h P(h)^u \mathcal{Q}_{[\eps,\eps']}(h)^{1-u}&=\frac{1}{\mathcal{Q}([\eps,\eps'])^{1-u}}\sum_h P(h)^u \mathcal{Q}(h,[\eps,\eps'])^{1-u}\notag\\
&\geq\frac{1}{\mathcal{Q}([\eps,\eps'])^{1-u}}\sum_h \frac{P(h)^u \mathcal{Q}(h,[\eps,\eps'])}{\mathcal{Q}(h)^u}\notag\\
&=\frac{1}{\mathcal{Q}([\eps,\eps'])^{1-u}}\sum_h \left(\frac{P(h)^u}{\mathcal{Q}(h)^u}\sum_{r:\sum_j |x_{r_j,j}-y_j|\in[\eps,\eps']}\mathcal{Q}(h,r) \right) \label{eq:pq-inequality}\end{align}

%, and, maybe, $L_j$ the $\ell_1$ distance contribution of the outcome of the $j^{\textrm{th}}$ coin flip. 
Then we can easily compute $\mathcal{Q}(r,h)=\prod_j q_{r_j,j} \,poi(k x_{r_j,j},h_j)$; $P(h)=\prod_j poi(k y_j,h_j)$; $\mathcal{Q}(h)=\prod_j (q_{1,j} poi(k x_{1,j},h_j)+q_{2,j} poi(k x_{2,j},h_j))$. Thus the right hand side of Equation~\ref{eq:pq-inequality} becomes, after pulling the $r$ sum to the outside, and pulling the product over $j$ before the sum over $h$:

\[\frac{1}{\mathcal{Q}([\eps,\eps'])^{1-u}} \sum_{r:\sum_j |x_{r_j,j}-y_j|\in[\eps,\eps']}  \prod_j \sum_{h_j} \frac{poi(k y_j,h_j)^u}{(q_1 poi(k x_{1,j},h_j)+q_2 poi(k x_{2,j},h_j))^u}q_{r_j,j}\,poi(k x_{r_j,j},h_j)\]
where the sum expression is seen to be exactly $\pi_{r_j,j}$ from Definition~\ref{def:pi} as claimed.

This expression (ignoring the common $\frac{1}{\mathcal{Q}([\eps,\eps'])^{1-u}}$ factor out front) is seen to exactly compute the probability that the sum of the independent random coin-flips $X_1,\ldots,X_n$ is in the interval $[\eps,\eps']$. We can then apply standard Chernoff bounds, as desired.
\end{proof}

\begin{definition}
Let $ch_2$ be the ratio between the right hand side of Equation~\ref{eq:lower-lower} and its Chernoff bound, Equation~\ref{eq:lower-lower-chernoff}, for the $u$ defined as $\arg\min_{u\in [0,1]} \sum_h P(h)^u \mathcal{Q}_{[\eps,\eps']}(h)^{1-u}$
\end{definition}

Combining the above results:
\begin{lemma}\label{lem:ch1-ch2}
    The failure probability of any algorithm for distinguishing $Poi(k)$ samples from $P$ versus $\mathcal{Q}_{[\eps,\eps']}$ is at least $ch_1\cdot ch_2\cdot \frac{1}{\mathcal{Q}([\eps,\eps'])^{1-u}}\min_{s\geq 0} e^{-s\eps}\prod_j \left( e^{s|x_{1,j}-y_j|}\pi_{1,j}+ e^{s|x_{2,j}-y_j|}\pi_{2,j}\right)$ for $u=\arg\min_{u\in [0,1]} \sum_h P(h)^u \mathcal{Q}_{[\eps,\eps']}(h)^{1-u}$.
\end{lemma}

We bound $\mathcal{Q}([\eps,\eps'])\leq \mathcal{Q}(\leq\eps')$ and bound this with a standard Chernoff bound, as \[\mathcal{Q}(\leq\eps')\leq e^{-\eps' \alpha}\prod_j \left(q_j e^{\alpha|x_{1,j}-y_j|}+(1-q_j) e^{\alpha|x_{2,j}-y_j|}\right)\] where we will use Chernoff parameter $\alpha$, with $\alpha$ from the optimization in Lemma~\ref{lem:upper-final}. Since Chernoff bounds are upper bounds, and $\mathcal{Q}([\eps,\eps'])$ appears in the denominator of the expression of Lemma~\ref{lem:ch1-ch2}, we thus have a lower bound for the failure probability. %\paul{Actually this is backwards! For $\alpha$ to be the chernoff parameter, we need $\alpha\geq 0$, but for the optimization, we need $\alpha<0$???}

The expression from Lemma~\ref{lem:ch1-ch2} is lower-bounded by its minimum over all $u$. Thus, from Lemma~\ref{lem:ch1-ch2}, expanding out the expressions for $\pi$ from Definition~\ref{def:pi} and substituting in the above Chernoff bound for $\mathcal{Q}([\eps,\eps'])$ we have:

\begin{corollary}\label{cor:ch1-ch2}
    The failure probability of any algorithm for distinguishing $Poi(k)$ samples from $P$ versus $\mathcal{Q}_{[\eps,\eps']}$ is at least 
    
\iffalse    \begin{align}&ch_1\cdot ch_2\cdot \min_{u\in[0,1],s\geq 0}\frac{1}{\mathcal{Q}([\eps,\eps'])^{1-u}}e^{-s\eps}\prod_j \left( e^{s|x_{1,j}-y_j|}\pi_{1,j}+ e^{s|x_{2,j}-y_j|}\pi_{2,j}\right)\notag\\
%    &\hspace{-2cm}\geq ch_1\cdot ch_2\cdot \min_{u\in[0,1],s\geq 0} e^{-(s+(1-u)\alpha)\eps}\prod_j \frac{\sum_{i} \frac{poi(k\,y_j,i)^u}{(q_1 poi(k\,x_{1,j},i)+q_2 poi(k\,x_{2,j},i))^u}\left(q_j\,poi(k\,x_{1,j},i)e^{s|x_{1,j}-y_j|}+(1-q_j)\,poi(k\,x_{2,j},i)e^{s|x_{2,j}-y_j|}\right)}{\left(q_j e^{\alpha|x_{1,j}-y_j|}+(1-q_j) e^{\alpha|x_{2,j}-y_j|}\right)^{1-u}}\\
        &\hspace{-2.3cm}\geq ch_1\cdot ch_2\cdot\min_{u\in[0,1],s\geq 0} e^{-s\eps+(1-u)\alpha\eps'}\prod_j \sum_{i} \frac{poi(k\,y_j,i)^u \left(q_j\,poi(k\,x_{1,j},i)e^{s|x_{1,j}-y_j|}+(1-q_j)\,poi(k\,x_{2,j},i)e^{s|x_{2,j}-y_j|}\right)}{(q_1 poi(k\,x_{1,j},i)+q_2 poi(k\,x_{2,j},i))^u \left(q_j e^{\alpha|x_{1,j}-y_j|}+(1-q_j) e^{\alpha|x_{2,j}-y_j|}\right)^{1-u}}\notag\\
        &
\fi
\begin{gather}       \hspace{-2cm} ch_1\cdot ch_2\cdot e^{\alpha(\eps'-\eps)}\cdot \min_{u\in[0,1],s\geq 0} e^{\eps(-s+(1-u)\alpha)}\prod_j \sum_{i} \frac{poi(k\,y_j,i)^u \left(q_j\,poi(k\,x_{1,j},i)e^{s|x_{1,j}-y_j|}+(1-q_j)\,poi(k\,x_{2,j},i)e^{s|x_{2,j}-y_j|}\right)}{(q_1 poi(k\,x_{1,j},i)+q_2 poi(k\,x_{2,j},i))^u \left(q_j e^{\alpha|x_{1,j}-y_j|}+(1-q_j) e^{\alpha|x_{2,j}-y_j|}\right)^{1-u}} \label{eq:lower-optimization}\end{gather}
\end{corollary}
%\paul{I'm being inconsistent about whether $j$ indexes domain elements, or equivalence classes of domain elements (namely, whether $h_{y_j}$ and synonyms show up)}

Finally, we claim that, except for the Chernoff factors $ch_1,ch_2$ and the factor $e^{\alpha(\eps'-\eps)}$, Equation~\ref{eq:lower-optimization} is exactly equal to our upper bound, as expressed in either Lemma~\ref{lem:upper-final} or Equation~\ref{eq:relaxed}. We will use the fact that Equation~\ref{lem:derivative-analysis} below was optimized in Lemma~\ref{lem:upper-final} to show how to optimize this related equation with respect to $s$ (deferring for the moment the optimization over $u$).

\begin{lemma}\label{lem:derivative-analysis}
We claim that if $u\in(0,1)$ and the derivative of this expression \begin{equation}\label{eq:bound-without-conditioning2}\log e^{\eps (1-u) \alpha}\prod_j \left(\frac{\sum_i \poii{y_j}^u(q_j \poii{x_{1,j}}+(1-q_j)\poii{x_{2,j}})^{1-u}}{(q_j e^{|x_{1,j}-y_j| \alpha}+(1-q_j) e^{|x_{2,j}-y_j|\alpha})^{1-u}}\right)\end{equation} with respect to any $q_j$ is 0, and the derivative with respect to $\alpha$ is also 0, then the below expression %(which is analogous (???) to a Chernoff bound on $\sum_F P(F)^t Q(F,\geq\eps)^{1-t}$ with Chernoff parameter $s$) 
when minimized over $s$ achieves its global minimum at $s=0$:

\begin{equation}\label{eq:fake-chernoff2}\hspace{-.5cm}\log e^{-s\eps}\prod_j \left(\sum_i \left(\frac{\poii{y_j}}{(q_j \poii{x_{1,j}}+(1-q_j)\poii{x_{2,j}})}\right)^u(q_j e^{s |x_{1,j}-y_j|}\poii{x_{1,j}}+(1-q_j)e^{s |x_{2,j}-y_j|}\poii{x_{2,j}})\right)\end{equation}
\end{lemma}

\begin{proof}
Equation~\ref{eq:fake-chernoff2} is clearly convex in $s$, since log-convexity is preserved under sums and products; we thus show its $s$ derivative is 0 at $s=0$ to finish the proof of global optimality. The proof is ultimately straightforward, where we find expressions for the derivatives of Equation~\ref{eq:bound-without-conditioning2}, and find the right linear combination of them to imply that the $s$ derivative of Equation~\ref{eq:fake-chernoff2} is 0 at $s=0$.

The $q_j$ derivative of Equation~\ref{eq:bound-without-conditioning2}, which equals 0, is

\begin{align*}
    \hspace{-2cm}(1-u)\left(\frac{\sum_i \left(\frac{\poii{y_j}}{q_j(\poii{x_{1,j}}+(1-q_j)\poii{x_{2,j}}}\right)^u(\poii{x_{1,j}}-\poii{x_{2,j}})}{\sum_i\left(q_j poi( x_{1,j},i)+(1-q_j)poi( x_{2,j},i)\right)^{1-u}\left(poi( y_j,i)\right)^{u}}-\frac{(e^{|x_{1,j}-y_j| \alpha}- e^{|x_{2,j}-y_j|\alpha})}{(q_j e^{|x_{1,j}-y_j| \alpha}+(1-q_j) e^{|x_{2,j}-y_j|\alpha})}\right)
\end{align*}
which implies, since $u\in(0,1)$, that %(??? $u \neq 1$???)
\begin{equation}
\label{eq:derivative-of-qj-equals-0}
    \frac{\sum_i \left(\frac{\poii{y_j}}{q_j(\poii{x_{1,j}}+(1-q_j)\poii{x_{2,j}}}\right)^u(\poii{x_{1,j}}-\poii{x_{2,j}})}{\sum_i\left(q_j poi( x_{1,j},i)+(1-q_j)poi( x_{2,j},i)\right)^{1-u}\left(poi( y_j,i)\right)^{u}}=\frac{(e^{|x_{1,j}-y_j| \alpha}- e^{|x_{2,j}-y_j|\alpha})}{(q_j e^{|x_{1,j}-y_j| \alpha}+(1-q_j) e^{|x_{2,j}-y_j|\alpha})}
\end{equation}

On the other hand, the $\alpha$ derivative of Equation~\ref{eq:fake-chernoff2}, which is equals to $0$, is
\begin{equation*}
\epsilon(1-u) - (1-u)\left(\sum_j \frac{(q_j|x_{1,j}-y_j|e^{|x_{1,j}-y_j| \alpha}+(1-q_j) |x_{2,j} - y_j| e^{|x_{2,j}-y_j|\alpha})}{(q_j e^{|x_{1,j}-y_j| \alpha}+(1-q_j) e^{|x_{2,j}-y_j|\alpha})}\right)
\end{equation*}
Therefore, since $u\in(0,1)$, we have
\begin{equation}
\label{eq:derivative-of-alpha}
\sum_j  \frac{(q_j|x_{1,j}-y_j|e^{|x_{1,j}-y_j| \alpha}+(1-q_j) |x_{2,j} - y_j| e^{|x_{2,j}-y_j|\alpha})}{(q_j e^{|x_{1,j}-y_j| \alpha}+(1-q_j) e^{|x_{2,j}-y_j|\alpha})} = \epsilon
\end{equation}

Going back to Equation~\ref{eq:fake-chernoff2}, the $s$ derivative of the log of the $j^\textrm{th}$ term of Equation~\ref{eq:fake-chernoff2}, evaluated at $s=0$ is
\begin{equation*}
    \frac{\sum_i \left(\frac{\poii{y_j}}{q_j(\poii{x_{1,j}}+(1-q_j)\poii{x_{2,j}}}\right)^u (q_j |x_{1,j}-y_j|\poii{x_{1,j}} +(1-q_j) |x_{2,j}-y_j|\poii{x_{2,j}})}{\sum_i \poii{y_j}^u(q_j \poii{x_{1,j}}+(1-q_j)\poii{x_{2,j}})^{1-u}}
\end{equation*}

Notice that the last term in the numerator can be decomposed as \[\hspace{-2cm}(q_j |x_{1,j}-y_j|\poii{x_{1,j}} +(1-q_j) |x_{2,j}-y_j|\poii{x_{2,j}})=A(\poii{x_{1,j}}-\poii{x_{2,j}})+B(q_j \poii{x_{1,j}}+(1-q_j)\poii{x_{2,j}})\]
for $A=q_j(1-q_j)(|x_{1,j}-y_j|-|x_{2,j}-y_j|) \textrm{ and } B=q_j|x_{1,j}-y_j|+(1-q_j)|x_{2,j}-y_j|$.

Substituting back to the $s$ derivative of the log of the $j^\textrm{th}$ term of Equation~\ref{eq:fake-chernoff2}, evaluated at $s=0$, we have

\begin{align*}
&\frac{\sum_i \left(\frac{\poii{y_j}}{q_j(\poii{x_{1,j}}+(1-q_j)\poii{x_{2,j}}}\right)^u A(\poii{x_{1,j}}-\poii{x_{2,j}})}{\sum_i \poii{y_j}^u(q_j \poii{x_{1,j}}+(1-q_j)\poii{x_{2,j}})^{1-u}} \\
& \quad + \frac{\sum_i \left(\frac{\poii{y_j}}{q_j(\poii{x_{1,j}}+(1-q_j)\poii{x_{2,j}}}\right)^u B(q_j \poii{x_{1,j}}+(1-q_j)\poii{x_{2,j}})}{\sum_i \poii{y_j}^u(q_j \poii{x_{1,j}}+(1-q_j)\poii{x_{2,j}})^{1-u}} \\
&= A\frac{(e^{|x_{1,j}-y_j| \alpha}- e^{|x_{2,j}-y_j|\alpha})}{(q_j e^{|x_{1,j}-y_j| \alpha}+(1-q_j) e^{|x_{2,j}-y_j|\alpha})} + B \quad \eqref{eq:derivative-of-qj-equals-0} \\
&= \frac{((A+Bq_j)e^{|x_{1,j}-y_j| \alpha}+((1-q_j)B-A) e^{|x_{2,j}-y_j|\alpha})}{(q_j e^{|x_{1,j}-y_j| \alpha}+(1-q_j) e^{|x_{2,j}-y_j|\alpha})}\\
&= \frac{(q_j|x_{1,j}-y_j|e^{|x_{1,j}-y_j| \alpha}+(1-q_j) |x_{2,j} - y_j| e^{|x_{2,j}-y_j|\alpha})}{(q_j e^{|x_{1,j}-y_j| \alpha}+(1-q_j) e^{|x_{2,j}-y_j|\alpha})}
\end{align*}
which is exactly the $j$-th term of the left hand side of Equation~\ref{eq:derivative-of-alpha}. Hence, summing up the previous expression for all $j$ gives us $\epsilon$.
This shows that the derivative of $s$ at $s=0$ of Equation~\ref{eq:fake-chernoff2} is $-\epsilon + \epsilon = 0$, concluding the proof.
\end{proof}

The next lemma concludes this section, summarizing that, except for the slack in the Chernoff bounds, and our choice of $\eps'>\eps$, our upper and lower bounds match.

\begin{lemma}\label{lem:lower-matching-with-3-factors}
Under the conditions of Lemma~\ref{lem:upper-final}---namely, when the testing upper bound of Equation~\ref{eq:merged-split} is optimized---then we have a lower bound for any testing algorithm that equals our upper bound times $ch_1\cdot ch_2\cdot e^{\alpha(\eps'-\eps)}$.
\end{lemma}
\begin{proof}
Corollary~\ref{cor:ch1-ch2} provides our testing lower bound, since Lemma~\ref{lem:upper-final} guarantees its conditions are satisfied. We now show that the minimization expression in Equation~\ref{eq:lower-optimization} (namely, without the initial $ch_1\cdot ch_2\cdot e^{\alpha(\eps'-\eps)}$ terms) exactly equals our testing upper bound---as analyzed in Lemma~\ref{lem:upper-final}. 

From Corollary~\ref{cor:ch1-ch2}, the minimization expression in Equation~\ref{eq:lower-optimization} without the initial $e^{(1-u)\alpha}$ term and without the second term in the denominator, is minimized over $s$ when $s=0$; thus the min over $u$ and $s$ is just the min over $u$ of the corresponding expression when $s=0$ is substituted, namely \[\min_{u\in[0,1]} e^{\eps(1-u)\alpha}\prod_j \sum_{i} \frac{poi(k\,y_j,i)^u \left(q_j\,poi(k\,x_{1,j},i)+(1-q_j)\,poi(k\,x_{2,j},i)\right)^{1-u}}{ \left(q_j e^{\alpha|x_{1,j}-y_j|}+(1-q_j) e^{\alpha|x_{2,j}-y_j|}\right)^{1-u}}\]

The logarithm of the expression being minimized is exactly Equation~\ref{eq:opt-reformulated}, which Lemma~\ref{lem:upper-final} guarantees has $u$ derivative equal to 0; further, the expression being minimized is log-convex in $u$ since log-convexity is preserved under both sums and products. Thus its global minimum is attained for the $u$ from the upper bound optimization, and its global minimum equals our upper bound, as expressed in Equation~\ref{eq:opt-reformulated} or Equation~\ref{eq:relaxed}.
\end{proof}

\section{G\"{a}rtner-Ellis Theorem}\label{sec:gartner-ellis}
We introduce a special version of G\"{a}rtner-Ellis Theorem, to analyze the convergence of Chernoff bounds when the number of domain elements $n$ goes to infinity. 

The G\"{a}rtner-Ellis theorem is a generalization of Cram\'{e}r's theorem, where Cram\'{e}r's theorem essentially says that ``in the limit as $n\rightarrow\infty$, Chernoff bounds on the mean of i.i.d. variables are tight, to $1+o(1)$ factors in the exponent. If $Z_n$ denotes the mean of $n$ independent copies of some real-valued random variable, and letting \[\Lambda_n(\lambda) = \log \Exp[e^{\lambda Z_n}]\] denote its corresponding moment generating function, then the log tail probabilities $\Pr[Z_n\geq x]$ are bounded by Chernoff bounds as $\min_{\lambda\geq 0} \Lambda_n(\lambda)-\lambda x=\min_{\lambda\geq 0} \Lambda_n(n\lambda)-n\lambda x$. Since we assume that $Z_n$ is the mean of i.i.d. random variables, then $\frac{1}{n}\Lambda_n(n\lambda)$ is identical for all $n$, and we may instead represent this as $\Lambda(\lambda)$, leading to a log Chernoff bound of $n\min_{\lambda\geq 0} \Lambda(\lambda)-\lambda x=-n\Lambda^*(x)$, where we define \[\Lambda^*(x) =\sup_{\lambda \in \mathbb{R}} \lambda x - \Lambda(\lambda)\] to be the Legendre-Fenchel transform of $\Lambda(\lambda)$. Cram\'{e}r's theorem says that this bound is tight in the limit as $n\rightarrow\infty$, where we scale the log probability by $\frac{1}{n}$ so that the right hand side does not go to infinity with $n$. Namely, while $\log \Pr[Z_n\geq x]\leq -n\Lambda^*(x)$ for any $n$, by Chernoff bounds, Cram\'{e}r's theorem says that the limit is tight, in the sense that $\lim_{n\rightarrow\infty}\frac{1}{n}\log \Pr[Z_n\geq x]= -\Lambda^*(x)$.

The G\"{a}rtner-Ellis theorem generalizes this to non-independent random variables. Namely, $Z_n$ is no longer restricted to be the mean of $n$ i.i.d. random variables but can now be any distribution, depending on $n$ in an almost arbitrary way---provided that the log moment generating function has a limit as $n\rightarrow\infty$, and avoids a few other pathologies as specified in Theorem~\ref{thm:G-E}.

%This whole part is thoroughly based on the Exercise 2.3.20 of the textbook~\cite{LDbook}, and we omit all the proofs. The audience who are interested in the proofs can check the textbook for details.
\iffalse
Consider a sequence of random variables $Z_n \sim p_n$ and $\lambda \in \mathbb{R}$, then the logarithmic moment generating function of $Z_n$ is the following:
\begin{equation}
    \Lambda_n(\lambda) = \log \Exp[e^{\lambda Z_n}]
\end{equation}
\fi

\begin{assumption}
\label{asp:G-E}
For each $\lambda \in \mathbb{R}$, the logarithmic moment generating function has a limit
\begin{equation}
    \Lambda(\lambda) = \lim_{n\rightarrow \infty} \frac{1}{n} \Lambda_n(n\lambda)
\end{equation}
which is finite and differentiable.
%exists as an extended real number. Further, $0$ belongs to the interior of $D_{\lambda} = \{\lambda\in\mathbb{R}: \Lambda(\lambda) < \infty\}$.
\end{assumption}
Then we can introduce the G\"{a}rtner-Ellis Theorem as follows:
\begin{theorem}[special case of G\"{a}rtner-Ellis Theorem]
\label{thm:G-E}
    If Assumption~\ref{asp:G-E} holds for $Z_{n}$, and if $x$ is not a ``discrete point", in the sense that $\Lambda^*$ is continuous at $x$ % $\overline{\lim}_{n\rightarrow\infty} \Pr[Z_n=x]=0$ %and $\Lambda(\lambda)$ is finite and differentiable everywhere, 
    then we have:
    \begin{equation}
        \lim_{n\rightarrow\infty} \frac{1}{n} \log \Pr[Z_n > x] = \lim_{n\rightarrow\infty} \frac{1}{n} \log \Pr[Z_n \geq x] = - \Lambda^*(x)
    \end{equation}
    if $\Lambda^*$ is non-increasing at $x$, and otherwise
        \begin{equation}
        \lim_{n\rightarrow\infty} \frac{1}{n} \log \Pr[Z_n < x] = \lim_{n\rightarrow\infty} \frac{1}{n} \log \Pr[Z_n \leq x] = - \Lambda^*(x)
    \end{equation}
    %Here $\Lambda^*(x) =\sup_{\lambda \in \mathbb{R}} \lambda x - \Lambda(\lambda) $, which is the Legendre-Fenchel transform of $\Lambda(\lambda)$.
\end{theorem}%\paul{not quite the ``large deviation principle"; we should edit this, but I don't want to make this complicated and opaque}

\subsection{Identity testing in the limit $n\rightarrow\infty$}

While we show upper and lower bounds for identity testing, we ultimately would like to compare these bounds to say that our tester is near-optimal. The G\"{a}rtner Ellis theorem provides a tool for doing this, in the limit as $n\rightarrow\infty$. Unfortunately, the G\"{a}rtner Ellis theorem does not give much insight into the rate of convergence, and a more targeted investigation of when ``Chernoff bounds are tight" seems warranted. Nonetheless, we view the G\"{a}rtner Ellis theorem as saying, essentially, that we should expect our Chernoff bounds to be tight for finite sample cases, since asymptotically they are tight.

We specifically consider a limit where we send $n\rightarrow\infty$ and $k\rightarrow\infty$ proportionally to $n$, and where the upper bounds of Section~\ref{sec:upper} on the log failure probability are exactly proportional to $n$. Explicitly, for any hypothesis distribution $P$, we consider ``subdividing it by factor $s$" to mean the process where each domain element of $P$ is split into $s$ new domain elements each of $\frac{1}{s}$ the probability; this new distribution, which we denote $P^{sub(s)}$ will have $n\cdot s$ domain elements; to make up for the reduced mass of each domain elements, we will sample $Poi(k\cdot s)$ times instead of $Poi(k)$ times.

The basic claim is that the optimization of Section~\ref{sec:upper} does not depend on $s$, except for trivial scaling. Explicitly, consider Equation~\ref{eq:relaxed} as we modify $s$. The hypothesis histogram entries $\histp$ will scale proportionally to $s$, and we will thus choose to scale the histogram entries of the optimization variables $\histq$ to scale with $s$; the probability masses, as measured by $y_j,x$, scale inversely with $s$, and thus the Poisson parameters $k y_j$ or $k x$ are invariant to $s$; the optimization variables $\mathbf{c},t,t'$ will be invariant to $s$; the constraints will thus be invariant to $s$; thus the overall objective function, since it has a factor of $h$ in both the first and second term, will scale proportionally to $s$. Since this optimization computes an upper bound on the log failure probability of our tester, this means that the tester coefficients ($c_{i,j}$) are independent of $s$, and the failure probability gets raised to the $s$ power (thus decaying exponentially with $s$).

Thus $\frac{1}{s}$ times the log failure probability is independent of $s$; and our overall claim is that this exactly matches our lower bounds, in the limit $s\rightarrow\infty$ (meaning that $n,k\rightarrow\infty$). 

Recall that we constructed a lower bound parameterized by $\eps'$. We first show, by the G\"{a}rtner-Ellis theorem that, for each fixed $\eps'>\eps$, the two Chernoff factors $ch_1,ch_2$ vanish in the limit; then we take the limit $\eps'\rightarrow\eps$ so that the final remaining factor $e^{\alpha(\eps'-\eps)}$ vanishes, leading, by Lemma~\ref{lem:lower-matching-with-3-factors}, to the desired conclusion.

\begin{lemma}
Fixing a hypothesis distribution $P$, and a number of samples $k$, and considering the limit as $s\rightarrow\infty$ of $P^{sub(s)}$, while we take $Poi(k s)$ samples, then, for any fixed $\eps'>\eps$, we have $\lim_{s\rightarrow\infty} \frac{1}{s}\log ch_1(P^{sub(s)},k s,\eps,\eps')=0$ and $\lim_{s\rightarrow\infty} \frac{1}{s}\log ch_2(P^{sub(s)},k s,\eps,\eps')=0$
\end{lemma}

Thus by Lemma~\ref{lem:lower-matching-with-3-factors} we have that for any sequence of testers, on input $P^{sub(s)}$ as $s\rightarrow\infty$, their log failure probability, times $\frac{1}{s}$, must have lim inf greater than or equal to our upper bound plus $\alpha(\eps'-\eps)$, where $\alpha$ is negative, and is evaluated from our upper bound at $s=1$. 

Thus sending $\eps'\rightarrow\eps$ yields that any sequence of testers must have log failure probability times $\frac{1}{s}$ that has lim inf (as $s\rightarrow\infty$) greater than or equal to our upper bound. Since our upper bound is explicitly an upper bound on testing performance, then there is a sequence of testers---namely, with coefficients $c_{i,j}$, independent of $s$, such that $\frac{1}{s}$ times the log of its failure probability has the limit exactly as specified by our upper bound, and no testers can do any better in this limit. This yields the final part of our main result, Theorem~\ref{thm:main}.

\bibliography{identity}
\bibliographystyle{alpha}

\end{document}